\documentclass[11pt]{amsart}

\title[Non-uniqueness for the transport equation]{Non-uniqueness for the transport equation with Sobolev vector fields}
\date{\today}

\author{Stefano Modena}
\address{Institut f\"ur Mathematik, Universit\"at Leipzig, D-04109 Leipzig, Germany}
\email{stefano.modena@math.uni-leipzig.de}

\author{L\'aszl\'o Sz\'ekelyhidi Jr.}
\address{Institut f\"ur Mathematik, Universit\"at Leipzig, D-04109 Leipzig, Germany}
\email{laszlo.szekelyhidi@math.uni-leipzig.de}

\usepackage[active]{srcltx}
\usepackage{ifpdf}
\ifpdf 
    \usepackage[pdftex,     
            plainpages=false,   
            breaklinks=true,    
            colorlinks=true,
            pdftitle=My Document
            pdfauthor=My Good Self
           ]{hyperref} 
\else 
    \usepackage{hyperref}       
\fi

\usepackage{amsfonts,amsmath,esint, xfrac}	
\usepackage{amssymb}
\usepackage{verbatim}
\usepackage{amsopn}
\usepackage[english]{babel}
\usepackage{amsthm}
\usepackage{enumerate}
\usepackage{mathrsfs}	
\usepackage{mathtools}
\usepackage{color}
\usepackage{tikz}
\usepackage{soul}
\usepackage{url}
\usepackage[overload]{empheq}

\theoremstyle{definition} \newtheorem{definition}{Definition}[section]
\theoremstyle{definition} \newtheorem{remark}[definition]{Remark}
\theoremstyle{plain} \newtheorem{lemma}[definition]{Lemma}
\theoremstyle{plain} \newtheorem{proposition}[definition]{Proposition}
\theoremstyle{plain} \newtheorem{theorem}[definition]{Theorem}
\theoremstyle{plain} \newtheorem{corollary}[definition]{Corollary}
\theoremstyle{definition} 
\theoremstyle{plain} 
\theoremstyle{definition}

\DeclareMathOperator{\sign}{sign}
\newcommand{\R}{\mathbb{R}}

\newcommand{\N}{\mathbb{N}}
\newcommand{\Z}{\mathbb{Z}}
\newcommand{\T}{\mathbb{T}}

\renewcommand{\div}{\textrm{div }}
\newcommand{\e}{\varepsilon}

\newcommand{\supp}{\mathrm{supp} \ }

\newcommand*{\mytag}[2]{(\ref{#1})$_{#2}$}

\makeatletter
\def\grd@save@target#1{%
  \def\grd@target{#1}}
\def\grd@save@start#1{%
  \def\grd@start{#1}}
\tikzset{
  grid with coordinates/.style={
    to path={%
      \pgfextra{%
        \edef\grd@@target{(\tikztotarget)}%
        \tikz@scan@one@point\grd@save@target\grd@@target\relax
        \edef\grd@@start{(\tikztostart)}%
        \tikz@scan@one@point\grd@save@start\grd@@start\relax
        \draw[minor help lines] (\tikztostart) grid (\tikztotarget);
        \draw[major help lines] (\tikztostart) grid (\tikztotarget);
        \grd@start
        \pgfmathsetmacro{\grd@xa}{\the\pgf@x/1cm}
        \pgfmathsetmacro{\grd@ya}{\the\pgf@y/1cm}
        \grd@target
        \pgfmathsetmacro{\grd@xb}{\the\pgf@x/1cm}
        \pgfmathsetmacro{\grd@yb}{\the\pgf@y/1cm}
        \pgfmathsetmacro{\grd@xc}{\grd@xa + \pgfkeysvalueof{/tikz/grid with coordinates/major step}}
        \pgfmathsetmacro{\grd@yc}{\grd@ya + \pgfkeysvalueof{/tikz/grid with coordinates/major step}}
        \foreach \x in {\grd@xa,\grd@xc,...,\grd@xb}
        \node[anchor=north] at (\x,\grd@ya) {\pgfmathprintnumber{\x}};
        \foreach \y in {\grd@ya,\grd@yc,...,\grd@yb}
        \node[anchor=east] at (\grd@xa,\y) {\pgfmathprintnumber{\y}};
      }
    }
  },
  minor help lines/.style={
    help lines,
    step=\pgfkeysvalueof{/tikz/grid with coordinates/minor step}
  },
  major help lines/.style={
    help lines,
    line width=\pgfkeysvalueof{/tikz/grid with coordinates/major line width},
    step=\pgfkeysvalueof{/tikz/grid with coordinates/major step}
  },
  grid with coordinates/.cd,
  minor step/.initial=.2,           
  major step/.initial=1,            
  major line width/.initial=1pt,    
}
\makeatother

\begin{document}

\begin{abstract}
	We construct a large class of examples of non-uniqueness for the linear transport equation and the transport-diffusion equation with divergence-free vector fields in Sobolev spaces $W^{1, p}$.
\end{abstract}

\maketitle

\section{Introduction}

This paper concerns the problem of (non)uniqueness of solutions to the transport equation in the periodic setting
\begin{align}
\label{e:transport}
	\partial_t\rho+u\cdot \nabla\rho &=0,\\
\label{e:initial:cond}
	\rho_{|t=0}&=\rho^0
\end{align}
where $\rho:[0,T]\times \T^d\to\R$ is a scalar density, $u:[0,T]\times \T^d\to\R^d$ is a given vector field and $\T^d = \R^d / \Z^d$ is the $d$-dimensional flat torus. 

Unless otherwise specified, we assume in the following that $u \in L^1$ is \emph{incompressible}, i.e. 
\begin{equation}\label{e:incompressible}
	\div u=0
\end{equation}
in the sense of distributions. Under this condition, \eqref{e:transport} is formally equivalent to the continuity equation
\begin{equation}\label{eq:continuity}
\partial_t \rho + \div(\rho u) 	= 0.
\end{equation}

It is well known that the theory of classical solutions to \eqref{e:transport}-\eqref{e:initial:cond} is closely connected to the ordinary differential equation
\begin{equation}\label{e:ODE}
\begin{split}
	\partial_t X(t,x) &=u(t, X(t,x)),\\
	X(0,x) &= x,
\end{split}	
\end{equation}
via the formula $\rho(t, X(t,x))=\rho^0(x)$. In particular, for Lipschitz vector fields $u$ the well-posedness theory for \eqref{e:transport}-\eqref{e:initial:cond} follows from the Cauchy-Lipschitz theory for ordinary differential equations applied to \eqref{e:ODE}; on the other side, the  inverse flow map $\Phi(t) := X(t)^{-1}$ solves the transport equation
\begin{equation}
\label{eq:inverse:flow}
\begin{aligned}
\partial_t \Phi + (u \cdot \nabla) \Phi & = 0, \\
\Phi|_{t=0} & = \textrm{id}. 
\end{aligned}
\end{equation}

There are several PDE models, related, for instance, to fluid dynamics or to the theory of conservation laws (see for instance \cite{DiPerna:1989:Annals,Crippa:2015er,LeBris2008,Lions:1996vo,Lions:1998vp}), where one has to deal with vector fields which are not necessarily Lipschitz, but have lower regularity and therefore it is important to investigate the well-posedness of  \eqref{e:transport}-\eqref{e:initial:cond} in the case of non-smooth vector fields. 

Starting with the groundbreaking work of DiPerna-Lions \cite{DiPerna:1989vo} there is a wealth of well-posedness results for vector fields which are Sobolev or BV (we refer to the recent survey \cite{Ambrosio2017}, see also below) and in particular in recent years a lot of effort has been devoted to understanding how far the regularity assumptions can be relaxed. The main goal of this paper is to provide a lower bound on the regularity assumptions by showing, to our knowledge for the first time, that well-posedness can fail quite spectacularly even in the Sobolev setting, with $u \in C_t W^{1,\tilde p}_x := C([0,T]; W^{1,\tilde p}(\T^d))$ (see Theorem \ref{thm:main} for the precise statement). The mechanism we exploit to produce such ``failure of uniqueness'' is so strong that it can be applied also to the transport-diffusion equation
\begin{equation}
\label{eq:transport-diffusion}
\partial_t \rho + \div(\rho u) = \Delta \rho
\end{equation}
thus producing Sobolev vector fields $u \in C_t W^{1, \tilde p}_x$ for which uniqueness of solutions to \eqref{eq:transport-diffusion}-\eqref{e:initial:cond} fails in the class of densities $\rho \in C_t L^p$ (see Theorem \ref{thm:diffusion}). 

Both theorems can be generalized as follows: we can construct vector fields with arbitrary large regularity $u \in W^{\tilde m,\tilde p}$, $\tilde m \in \N$, for which uniqueness of solutions to \eqref{e:transport}-\eqref{e:initial:cond} or \eqref{eq:transport-diffusion}-\eqref{e:initial:cond} fails, in the class of densities $\rho \in W^{m, p}$, with arbitrary large $m \in \N$; moreover, we can do that even when on the r.h.s. of \eqref{eq:transport-diffusion} there is a higher order diffusion operator (see Theorems \ref{thm:strong} and \ref{thm:diffusion:higher}). 

Before stating the precise statements of these results, we present a brief (and far from complete) overview of the main well-posedness achievements present in the literature. We start with the analysis of the well-posedness for the transport equation in class of bounded densities, then we pass to the analysis of well-posedness for the transport equation in the class of $L^p$-integrable densities, with the statement of our Theorems \ref{thm:main} and \ref{thm:strong} and finally we discuss the transport-diffusion equation, with the statements of our Theorems \ref{thm:diffusion} and \ref{thm:diffusion:higher}. The last part of this introduction is devoted to a brief overview of the main techniques used in our proofs.

\subsection{The case of bounded densities}
\label{ss:bounded:densities}
The literature about \emph{rough} vector fields mainly concerns the well-posedness of \eqref{e:transport}-\eqref{e:initial:cond} in the class of bounded densities, $\rho \in L^\infty$. The reason for that can be found in the fact that the scientific community has been mainly interested in the well-posedness of ODE \eqref{e:ODE} and has used the PDE as a tool to attack the ODE problem: the general strategy is that a well-posedness result for the transport equation in the class of bounded densities yields a unique solution to the PDE \eqref{eq:inverse:flow} and thus one tries to prove that the flow $X(t) := \Phi(t)^{-1}$ is the unique meaningful solution, in the sense of \emph{regular Lagrangian flow}, to the ODE \eqref{e:ODE}. We refer to \cite{Ambrosio2017} for a precise definition of the notion of \emph{regular Lagrangian flow} and for a detailed discussion about the link between the Eulerian and the Lagrangian approach. 


Let us observe that for $\rho \in L^\infty$ the quantity $\rho u \in L^1$ and thus one can consider solutions to $\eqref{e:transport}$ (or, equivalently, to \eqref{eq:continuity}, since we are assuming incompressibility of the vector field) in distributional sense: $\rho$ is a \emph{distributional or weak solution} if 
\begin{equation}
\label{eq:weak:sln}
\int_0^T \int_{\T^d} \rho[\partial_t \varphi + u \cdot \nabla \varphi] dxdt= 0,
\end{equation}
for every $\varphi \in C^\infty_c ((0,T) \times \T^d)$. It is usually not difficult to prove existence of weak solutions, even if the vector field is very rough, taking advantage of the linearity of the equation. A much bigger issue is the uniqueness problem. 

The first result in this direction dates back to R. DiPerna and P.L. Lions \cite{DiPerna:1989vo} in 1989, when they proved uniqueness, in the class of bounded densities, for vector fields $u \in L^1_t W^{1,1}_x$ with bounded divergence. This result was extended in 2004 by L. Ambrosio in \cite{Ambrosio:2004cva} to vector fields $u \in L^1_t BV_x \cap L^\infty$ and with bounded divergence (see also \cite{Colombini:2002wp,Colombini:2003wl}) and very recently by S.~Bianchini and P.~Bonicatto in \cite{Bianchini:2017vf} for vector fields $u \in L^1_t BV_x$ which are merely \emph{nearly incompressible} (see, for instance, \cite{Ambrosio2017} for a definition of \emph{nearly incompressibility}).

The proofs of these results are very subtle and involves several deep ideas and sophisticated techniques. We could however try to summarize the heuristic behind all of them as follows: (very) roughly speaking, a Sobolev or BV vector field $u$ is Lipschitz-like (i.e. $Du$ is bounded) on a large set and there is just a small ``bad'' set, where $Du$ is very large. On the big set where $u$ is ``Lipschitz-like'', the classical uniqueness theory applies. Non-uniqueness phenomena could thus occur only on the small ``bad'' set. Uniqueness of solutions in the class of bounded densities is then a consequence of the fact that a \emph{bounded} density $\rho$ can not ``see'' this bad set, or, in other words, cannot concentrate on this bad set. 

With this rough heuristic in mind it is also perhaps not surprising that the theory cited above is heavily measure-theoretic. Nevertheless, the well-posedness for the ODE \eqref{e:ODE} fundamentally relies on the analysis and well-posedness theory of the associated PDE \eqref{e:transport}. More precisely, in \cite{DiPerna:1989vo} DiPerna and Lions introduced the notion of \emph{renormalized solution}. One calls a density $\rho \in L^1_{tx}$ renormalized for \eqref{e:transport} (for given $u$), if for any $\beta\in L^{\infty}(\R)\cap C^1(\R)$ it holds
\begin{equation}\label{e:renormalized}
\partial_t\beta(\rho)+u\cdot \nabla \beta(\rho)=0
\end{equation}
in the sense of distributions. Analogously to entropy-conditions for hyperbolic conservation laws, \eqref{e:renormalized} provides additional stability under weak convergence. Key to the well-posedness theory is then showing that any bounded distributional solution $\rho$ of \eqref{e:transport} is renormalized - this is done by showing convergence of the commutator 
\begin{equation}\label{e:commutator}
(u\cdot\nabla\rho)_{\e}-u_{\e}\cdot\nabla\rho_\e\,\to 0
\end{equation}
arising from suitable regularizations.

As we mentioned, uniqueness at the PDE level in the class of bounded densities implies, in all the cases considered above, uniqueness at the ODE level (again in the sense of the \emph{regular Lagrangian flow}). On the other hand, based on a self-similar mixing example of M.~Aizenmann in \cite{Aizenman:1978tx}, N.~Depauw in \cite{Depauw:2003wl} constructed an example of non-uniqueness for weak solutions with $\rho\in L^{\infty}((0,T)\times\T^d)$ and $u\in L^1(\e,1;BV(\T^d))$ for any $\e>0$ but $u\notin L^1(0,1;BV(\T^d))$. This example has been revisited in \cite{Colombini:2003wl,Alberti:2014cx,Alberti:2016wq,Yao:2017bp}. It should be observed, though, that the phenomenon of non-uniqueness in such ``mixing'' examples is \emph{Lagrangian} in the sense that it is a consequence of the degeneration of the flow map $X(t,x)$ as $t\to 0$; in particular, once again, the link between \eqref{e:transport} and \eqref{e:ODE} is crucial. 

%
%


%
%
%

\subsection{The case of unbounded densities}
\label{ss:unbounded:densities}

There are important mathematical models, related, for instance, to the Boltzmann equation (see \cite{DiPerna:1989:Annals}), incompressible 2D Euler \cite{Crippa:2015er}, or to the compressible Euler equations, in which the density under consideration is not bounded, but it belongs just to some $L^\infty_t L^p_x$ space. It is thus an important question to understand the well-posedness of the Cauchy problem \eqref{e:transport}-\eqref{e:initial:cond} in such larger functional spaces.

As a first step, we observe that for a density $\rho \in L^\infty_t L^p_x$ and a field $u \in L^1_t L^1_x$, the product $\rho u$ is not well defined in $L^1$ and thus the notion of weak solution as in \eqref{eq:weak:sln} has to be modified. There are several possibilities to overcome this issue. We mention two of them: either we require that $u \in L^1_t L^{p'}_x$, where $p'$ is the dual H\"older exponent to $p$, or we consider a notion of solution which cut off the regions where $\rho$ is unbounded. Indeed, this second possibility is encoded in the notion of renormalized solution \eqref{e:renormalized}. 

The well-posedness theory provided by \eqref{e:renormalized} for bounded densities is sufficient for the existence of a regular Lagrangian flow, which in turn leads to \emph{existence} also for unbounded densities. For the \emph{uniqueness}, an additional integrability condition is required:

\begin{theorem}[DiPerna-Lions \cite{DiPerna:1989vo}] 
\label{thm:di:perna:lions}
Let $p,\tilde p\in [1,\infty] $ and let $u\in L^1(0,T;W^{1,\tilde p}(\T^d))$ be a vector field with $\div u=0$. For any $\rho^0\in L^p(\T^d)$ there exists a unique renormalized solution of \eqref{e:transport}-\eqref{e:initial:cond}, satisfying $\rho\in C([0,T];L^p(\T^d))$. Moreover, if 
\begin{equation}\label{e:DP}
\frac{1}{p}+\frac{1}{\tilde p}\leq 1
\end{equation}
then this solution is \emph{unique} among all weak solutions with $\rho\in L^{\infty}(0,T;L^p(\T^d))$.
\end{theorem}

As we have already observed for the case of bounded densities, also in this more general setting existence of weak and renormalized solutions is not a difficult problem. It is as well not hard to show uniqueness in the class of renormalized solutions, using the fact that renormalized solutions in $L^\infty_t L^p_x$ have constant in time $L^p$ norm (it suffices to choose $\beta$ as a bounded smooth approximation of $\tau \mapsto |\tau|^p$). 

The crucial point in Theorem \ref{thm:di:perna:lions} concerns the uniqueness of the renormalized solution among all the weak solutions in $L^\infty_t L^p_x$, provided \eqref{e:DP} is satisfied. The reason why such uniqueness holds can be explained by the same heuristic as in the case of bounded densities: a vector field in $W^{1,\tilde p}$ is ``Lipschitz-like'' except on a small bad set, which can not ``be seen'' by a density in $L^p$, if \eqref{e:DP} holds, i.e. if $p$, although it is less than $\infty$, is sufficiently large w.r.t. $\tilde p$. On the more technical side, the integrability condition \eqref{e:DP} is necessary in the proof in \cite{DiPerna:1989vo} to show convergence of the commutator \eqref{e:commutator} in $L^1_{loc}$. 

The following question is therefore left open: does uniqueness of weak solutions hold in the class of densities $\rho \in L^\infty_t L^p_x$ for a vector field in $L^1_t L^{p'}_x \cap L^1_t W^{1, \tilde p}_x$, when \eqref{e:DP} fails?  

In a recent note \cite{Caravenna:2016kg} L.~Caravenna and G.~Crippa addressed this issue for the case $p=1$ and $\tilde p>1$, announcing the result that uniqueness holds under the additional assumption that $u$ is continuous. 

In this paper we show that if 
\begin{equation}
\label{eq:p:ineq}
\frac{1}{p} + \frac{1}{\tilde p} > 1 + \frac{1}{d-1}.
\end{equation}
then, in general, uniqueness fails. We remark that the Sobolev regularity of the vector field $u \in L^1_t W^{1,\tilde p}_x$ implies the existence of a unique regular Lagrangian flow (see in particular \cite{Ambrosio2015}). Nevertheless, quite surprisingly, our result shows that such Lagrangian uniqueness is of no help to get uniqueness on the Eulerian side. 

Previously, examples of such Eulerian non-uniqueness have been constructed, for instance, in \cite{Crippa:2014ta}, based on the method of convex integration from \cite{DeLellis:2009jh}, yielding merely bounded velocity $u$ and density $\rho$. However, such examples do not satisfy the differentiability condition $u \in W^{1, \tilde p}$ for any $\tilde p\geq 1$ and therefore do not possess an associated Lagrangian flow.

Here is the statement of our first and main result. 

\begin{theorem}
\label{thm:main}
Let $\e>0$, $\bar \rho \in C^\infty([0,T] \times \T^d)$, with
\begin{equation*}
\int_{\T^d} \bar \rho(0,x) dx = \int_{\T^d} \bar \rho(t,x) dx \text{ for every } t \in [0,T].
\end{equation*} 
Let $p \in (1, \infty)$, $\tilde p \in [1, \infty)$ such that \eqref{eq:p:ineq} holds.
Then there exist $\rho : [0,T] \times \T^d \to \R$, $u:[0,T] \times \T^d \to \R^d$ such that
\begin{enumerate}[(a)]
\item $\rho \in C\bigl([0,T]; L^p (\T^d)\bigr)$, $u \in C\bigl([0,T]; W^{1,\tilde p} (\T^d)\cap L^{p'} (\T^d)\bigr)$;
\item $(\rho, u)$ is a weak solution to \eqref{e:transport} and \eqref{e:incompressible};
\item at initial and final time $\rho$ coincides with $\bar \rho$, 
i.e. 
\begin{equation*}
\rho(0,\cdot) = \bar \rho(0, \cdot),  \quad \rho(T, \cdot) = \bar \rho(T, \cdot);
\end{equation*}
\item $\rho$ is $\e$-close to $\bar \rho$ i.e.
\begin{equation*}
\begin{split}
\sup_{t \in [0,T]} \big\|\rho(t,\cdot) - \bar \rho(t, \cdot)\big\|_{L^p(\T^d)} 				& \leq \e. 
\end{split}
\end{equation*}
\end{enumerate}
\end{theorem}

Our theorem has the following immediate consequences.
\begin{corollary}[Non-uniqueness]
\label{cor:nonuniq}
Assume \eqref{eq:p:ineq}. Let $\bar \rho \in C^\infty(\T^d)$ with $\int_{\T^d}\bar\rho\,dx=0$. Then there exist
\begin{equation*}
\rho \in C\big([0,T]; L^p(\T^d)\big), \quad u \in C\big([0,T]; W^{1,\tilde p} (\T^d)\cap L^{p'} (\T^d))\big)
\end{equation*}
such that $(\rho,u)$ is a weak solution to \eqref{e:transport}, \eqref{e:incompressible}, and $\rho \equiv 0$ at $t=0$, $\rho \equiv \bar \rho$ at $t=T$.
\end{corollary}

\begin{proof}
Let $\chi: [0,T] \to \R$ such that $\chi \equiv 0$ on $[0,T/4]$, $\chi \equiv 1$ on $[3T/4, T]$. Apply Theorem \ref{thm:main} with $\bar \rho(t,x) := \chi(t) \bar \rho(x)$.
\end{proof}

\begin{corollary}[Non-renormalized solution]
\label{cor:nonrem}
Assume \eqref{eq:p:ineq}. 
Then there exist
\begin{equation*}
\rho \in C\big([0,T]; L^p(\T^d)\big), \quad u \in C\big([0,T]; W^{1,\tilde p} (\T^d)\cap L^{p'} (\T^d))\big)
\end{equation*}
such that $(\rho,u)$ is a weak solution to \eqref{e:transport}, \eqref{e:incompressible}, and $\|\rho(t)\|_{L^p(\T^d)}$ is not constant in time.
\end{corollary}
\begin{proof}
Take a smooth map $\bar \rho(t,x)$ such that its spatial mean value is constant in time, but its $L^p$ norm is not constant in time. Apply Theorem \ref{thm:main} with such $\bar \rho$ and
\begin{equation*}
\e := \frac{1}{4} \max_{t,s} \bigg| \|\rho(t)\|_{L^p(\T^d)} - \|\rho(s)\|_{L^p(\T^d)} \bigg|. 
\end{equation*}
\end{proof}

\begin{remark}
\label{r:main:thm}
We list some remarks about the statement of the theorem.
\begin{enumerate}[1.]
\item Condition \eqref{eq:p:ineq} implies that $d \geq 3$. In fact it is not clear if a similar statement could hold for $d=2$ - see for instance \cite{Alberti:2014dl} for the case of autonomous vector fields.
\item Our theorem shows the optimality of the condition of DiPerna-Lions in \eqref{e:DP}, at least for sufficiently high dimension $d\geq 3$. 
\item The requirement that $\bar \rho$ has constant (in time) spatial mean value is necessary because weak solutions to \eqref{e:transport}, \eqref{e:incompressible} preserve the spatial mean.
\item The condition \eqref{eq:p:ineq} implies that the $L^{p'}$-integrability of the velocity $u$ does not follow from the Sobolev embedding theorem.
\item We expect that the statement of Theorem \ref{thm:main} remains valid if \eqref{eq:p:ineq} is replaced by 
\begin{equation}\label{e:p:ineq:sharp}
	\frac{1}{p}+\frac{1}{\tilde p}>1+\frac{1}{d}.
\end{equation}
It would be interesting to see if this condition is sharp in the sense that uniqueness holds provided 
\begin{equation*}
\frac{1}{p} + \frac{1}{\tilde p} \leq 1 + \frac{1}{d}\,.
\end{equation*}
In this regard we note that \eqref{e:p:ineq:sharp} implies $\tilde p<d$. Conversely, if $u\in W^{1,\tilde p}$ with $\tilde p>d$, the Sobolev embedding implies that $u$ is continuous so that the uniqueness statement in \cite{Caravenna:2016kg} applies.
\item The given function $\bar \rho$ could be less regular than $C^\infty$, but we are not interested in following this direction here.
\item It can be shown that the dependence of $\rho,u$ on time is actually $C^\infty_t$, not just continuous, since we treat time just as a parameter. 
\end{enumerate}
\end{remark}

Inspired by the heuristic described above, the proof of our theorem is based on the construction of densities $\rho$ and vector fields $u$ so that $\rho$ is, in some sense, concentrated on the ``bad'' set of $u$, provided \eqref{eq:p:ineq} holds. To construct such densities and fields, we treat the linear transport equation \eqref{e:transport} as a non-linear PDE, whose unknowns are both $\rho$ and $u$: this allows us to control the interplay between density and field. More precisely, we must deal with two opposite needs: on one side, to produce ``anomalous'' solutions, we need to highly concentrate $\rho$ and $u$; on the other side, too highly concentrated functions fail to be Sobolev or even $L^p$-integrable. The balance between these two needs is expressed by \eqref{eq:p:ineq}. 

It is therefore possible to guess that, under a more restrictive assumption than \eqref{eq:p:ineq}, one could  produce anomalous solutions enjoying much more regularity than just $\rho \in L^p$ and $u \in W^{1, \tilde p}$. Indeed, we can produce anomalous solutions as regular as we like, as shown in the next theorem, where \eqref{eq:p:ineq} is replaced by \eqref{eq:strong:p:ineq}. 

\begin{theorem}
\label{thm:strong}
Let $\e>0$, $\bar \rho \in C^\infty([0,T] \times \T^d)$, with
\begin{equation*}
\int_{\T^d} \bar \rho(0,x) dx = \int_{\T^d} \bar \rho(t,x) dx \text{ for every } t \in [0,T].
\end{equation*} 
Let $p, \tilde p \in [1, \infty)$ and $m, \tilde m \in \N$ such that
\begin{equation}
\label{eq:strong:p:ineq}
\frac{1}{p} + \frac{1}{\tilde p} > 1 + \frac{m + \tilde m}{d-1}.
\end{equation}
Then there exist $\rho : [0,T] \times \T^d \to \R$, $u:[0,T] \times \T^d \to \R^d$ such that
\begin{enumerate}[(a)]
\item $\rho \in C([0,T], W^{m,p} (\T^d))$, $u \in C([0,T]; W^{\tilde m,\tilde p} (\T^d))$, $\rho u \in C([0,1]; L^1 (\T^d))$;
\item $(\rho, u)$ is a weak solution to \eqref{e:transport}, \eqref{e:incompressible};
\item at initial and final time $\rho$ coincides with $\bar \rho$, 
i.e. 
\begin{equation*}
\rho(0,\cdot) = \bar \rho(0, \cdot),  \quad \rho(T, \cdot) = \bar \rho(T, \cdot);
\end{equation*}
\item $\rho$ is $\e$-close to $\bar \rho$ i.e.
\begin{equation*}
\begin{split}
\sup_{t \in [0,T]} \big\|\rho(t,\cdot) - \bar \rho(t, \cdot)\big\|_{W^{m,p}(\T^d)} 		& \leq \e. \\
\end{split}
\end{equation*}

\end{enumerate}
\end{theorem}

\begin{remark}
The analogues of Corollaries \ref{cor:nonuniq} and \ref{cor:nonrem} continue to hold in Theorems \ref{thm:strong}. Observe also that \eqref{eq:strong:p:ineq} reduces to \eqref{eq:p:ineq} if we choose $m=0$ and $\tilde m =1$. 

\end{remark}

\begin{remark}
\label{rmk:ss'}
Contrary to Theorem \ref{thm:main}, here we do not show that $u \in C([0,T], L^{p'}(\T^d))$. Here we prove that $\rho u \in C([0,T], L^1(\T^d))$ by showing that $\rho \in C([0,T]; L^s(\T^d))$ and $u \in C([0,T]; L^{s'}(\T^d))$ for some suitably chosen $s,s' \in (1, \infty)$. This is also the reason why in Theorem \ref{thm:strong} we allow the case $p = 1$. Indeed, Theorem \ref{thm:main}, for any given $p$, produces a vector field $u \in C_t L^{p'}_x$; on the contrary, Theorem \ref{thm:strong} just produces a field $u \in C_t L^{s'}_x$, for some $s' < p'$.

\end{remark}

\subsection{Extension to the transport-diffusion equation}
\label{ss:transport-diffusion}

The mechanism of concentrating the density in the same set where the field is concentrated, used to construct anomalous solutions to the transport equation, can be used as well to prove non-uniqueness for the transport-diffusion equation \eqref{eq:transport-diffusion}. 

The diffusion term $\Delta \rho$ ``dissipates the energy'' and therefore, heuristically, it helps for uniqueness. Non-uniqueness can thus be caused only by the transport term $\div (\rho u)= u \cdot \nabla \rho$. Therefore, as a general principle, whenever a uniqueness result is available for the transport equation, the same result applies to the transport-diffusion equation (see, for instance, \cite{LeBris2004}, \cite{Levy:2016tl} and \cite{Crippa:2015er}). Moreover, the diffusion term $\Delta \rho$ is so strong that minimal assumptions on $u$ are enough to have uniqueness: this is the case, for instance, if $u$ is just bounded, or even $u \in L^r_t L^q_x$, with $2/r + d/q \leq 1$ (see \cite{Ladyzenskaja:1968} and also \cite{Bianchini2017}, where this relation between $r,q,d$ is proven to be sharp). Essentially, in this regime the transport term can be treated as a lower order perturbation of the heat equation. 

On the other hand, the technique we use to prove non-uniqueness for the transport equation allows us to construct densities and fields, whose concentrations are so high that the transport term ``wins'' over the diffusion one and produces anomalous solutions to \eqref{eq:transport-diffusion} as well. Roughly speaking, we have to compare $\div (\rho u)$ with $\Delta \rho = \div(\nabla \rho)$, or, equivalently, $\rho u$ with $\nabla \rho$, for instance in the $L^1$ norm. The way we construct concentration of $\rho$ and $u$ can be arranged, under a more restrictive assumption than \eqref{eq:p:ineq}, so that
\begin{equation*}
\|\rho u\|_{L^1} \approx 1, \qquad \|\nabla \rho\|_{L^1} \ll 1
\end{equation*}
(see the last inequality in \eqref{eq:mikado:est:1} and \eqref{eq:mikado:est:2:d}) and thus the transport term is ``much larger'' than the diffusion one. The precise statement is as follows.

\begin{theorem}\label{thm:diffusion}
Let $\e>0$, $\bar \rho \in C^\infty([0,T] \times \T^d)$, with
\begin{equation*}
\int_{\T^d} \bar \rho(0,x) dx = \int_{\T^d} \bar \rho(t,x) dx \text{ for every } t \in [0,T].
\end{equation*} 
Let $p \in (1, \infty)$, $\tilde p \in [1, \infty)$ such that
\begin{equation}
\label{eq:p:ineq:diffu}
\frac{1}{p} + \frac{1}{\tilde p} > 1 + \frac{1}{d-1},\quad
p' < d-1 \,. 
\end{equation}
Then there exist $\rho : [0,T] \times \T^d \to \R$, $u:[0,T] \times \T^d \to \R^d$ such that
\begin{enumerate}[(a)]
\item $\rho \in C\bigl([0,T]; L^p (\T^d)\bigr)$, $u \in C\bigl([0,T]; W^{1,\tilde p} (\T^d)\cap L^{p'} (\T^d)\bigr)$;
\item $(\rho, u)$ is a weak solution to \eqref{eq:transport-diffusion} and \eqref{e:incompressible};
\item at initial and final time $\rho$ coincides with $\bar \rho$, 
i.e. 
\begin{equation*}
\rho(0,\cdot) = \bar \rho(0, \cdot),  \quad \rho(T, \cdot) = \bar \rho(T, \cdot);
\end{equation*}
\item $\rho$ is $\e$-close to $\bar \rho$ i.e.
\begin{equation*}
\begin{split}
\sup_{t \in [0,T]} \big\|\rho(t,\cdot) - \bar \rho(t, \cdot)\big\|_{L^p(\T^d)} & \leq \e.
\end{split}
\end{equation*}
\end{enumerate}

\end{theorem}



As for the transport equation, also for \eqref{eq:transport-diffusion} we can generalize Theorem \ref{thm:diffusion}, to get densities and fields with arbitrary large regularity. Moreover, we can cover also the case of diffusion operators of arbitrary large order:
\begin{equation}
\label{eq:transport:diffusion:higher}
\partial_t \rho + \div(\rho u) = L \rho,
\end{equation} 
where $L$ is a constant coefficient differential operator of order $k \in \N$, $k \geq 2$, not necessarily elliptic. 

\begin{theorem}\label{thm:diffusion:higher}
Let $\e>0$, $\bar \rho \in C^\infty([0,T] \times \T^d)$, with
\begin{equation*}
\int_{\T^d} \bar \rho(0,x) dx = \int_{\T^d} \bar \rho(t,x) dx \text{ for every } t \in [0,T].
\end{equation*} 
Let $p, \tilde p \in [1, \infty)$ and $m, \tilde m \in \N$ such that
\begin{equation}
\label{eq:p:ineq:diffu:2}
\frac{1}{p} + \frac{1}{\tilde p} > 1 + \frac{m+\tilde m}{d-1},\quad 
\tilde p < \frac{d-1}{\tilde m + k -1}\,.
\end{equation}
Then there exist $\rho : [0,T] \times \T^d \to \R$, $u:[0,T] \times \T^d \to \R^d$ such that
\begin{enumerate}[(a)]
\item $\rho \in C\bigl([0,T]; W^{m,p} (\T^d)\bigr)$, $u \in C\bigl([0,T]; W^{\tilde m,\tilde p} (\T^d)\bigr)$, $\rho u \in C([0,1]; L^1(\T^d))$;
\item $(\rho, u)$ is a weak solution to \eqref{eq:transport:diffusion:higher} and \eqref{e:incompressible};
\item at initial and final time $\rho$ coincides with $\bar \rho$, 
i.e. 
\begin{equation*}
\rho(0,\cdot) = \bar \rho(0, \cdot),  \quad \rho(T, \cdot) = \bar \rho(T, \cdot);
\end{equation*}
\item $\rho$ is $\e$-close to $\bar \rho$ i.e.
\begin{equation*}
\begin{split}
\sup_{t \in [0,T]} \big\|\rho(t,\cdot) - \bar \rho(t, \cdot)\big\|_{W^{m,p}(\T^d)} & \leq \e.
\end{split}
\end{equation*}
\end{enumerate}

\end{theorem}

\begin{remark}
The analogues of Corollaries \ref{cor:nonuniq} and \ref{cor:nonrem} continue to hold in Theorems \ref{thm:diffusion} and \ref{thm:diffusion:higher}. Remark \ref{rmk:ss'} applies also to the statement of 
Theorem \ref{thm:diffusion:higher}. 

Observe also that, if we choose $m=0$, $\tilde m=1$, $k=2$, the first condition in \eqref{eq:p:ineq:diffu} reduces to the first condition in \eqref{eq:p:ineq:diffu:2}, nevertheless \eqref{eq:p:ineq:diffu} is not equivalent to \eqref{eq:p:ineq:diffu:2}. Indeed, \eqref{eq:p:ineq:diffu} implies \eqref{eq:p:ineq:diffu:2}, but the viceversa is not true, in general. This can be explained by the fact that Theorem \ref{thm:diffusion}, for any given $p$,  produces a vector field $u \in C_t L^{p'}_x$, while Theorem \ref{thm:diffusion:higher} just produces a field $u \in C_t L^{s'}_x$ for some $s' < p'$.

\end{remark}

\subsection{Strategy of the proof}
\label{ss:strategy:proof}

Our strategy is based on the technique of convex integration that has been developed in the past years for the incompressible Euler equations in connection with Onsager's conjecture, see \cite{DeLellis:2013im,DeLellis:2012tz,Buckmaster:2013vv,Buckmaster:2014ty,Buckmaster:2014th,Isett:2016to,Buckmaster:2017uz} and in particular inspired by the recent extension of the techniques to weak solutions of the Navier-Stokes equations in \cite{Buckmaster:2017wf}. 
Whilst the techniques that led to progress and eventual resolution of Onsager's conjecture in \cite{Isett:2016to} are suitable for producing examples with H\"older continuous velocity (with small exponent) \cite{Isett:2014uw}, being able to ensure that the velocity is in a Sobolev space $W^{1,\tilde p}$, i.e.~with one full derivative, requires new ideas. 

A similar issue appears when one wants to control the dissipative term $-\Delta u$ in the Navier-Stokes equations. Inspired by the theory of intermittency in hydrodynamic turbulence,  T.~Buckmaster and V.~Vicol in \cite{Buckmaster:2017wf} introduced ``intermittent Beltrami flows'', which are spatially inhomogeneous versions of the classical Beltrami flows used in \cite{DeLellis:2013im,DeLellis:2012tz,Buckmaster:2013vv,Buckmaster:2014ty,Buckmaster:2014th}. In contrast to the homogeneous case, these have different scaling for different $L^q$ norms at the expense of a diffuse Fourier support. In particular, one can ensure small $L^q$ norm for small $q>1$, which in turn leads to control of the dissipative term.

In this paper we introduce concentrations to the convex integration scheme in a different way, closer in spirit to the $\beta$-model, introduced by Frisch, Sulem and Nelkin \cite{Frisch:1978dv,FrischBook} as a simple model for intermittency in turbulent flows. In addition to a large parameter $\lambda$ that controls the frequency of oscillations, we introduce a second large parameter $\mu$ aimed at controlling concentrations. Rather than working in Fourier space, we work entirely in $x$-space and use ``Mikado flows", introduced in \cite{SzekelyhidiJr:2016tp} and used in \cite{Isett:2016to,Buckmaster:2017uz} as the basic building blocks. These building blocks consist of pairwise disjoint (periodic) pipes in which the divergence-free velocity  and, in our case, the density are supported. In particular, our construction only works for dimensions $d\geq 3$. The oscillation parameter $\lambda$ controls the frequency of the periodic arrangement - the pipes are arranged periodically with period $1/\lambda$. The concentration parameter $\mu$ controls the relative (to $1/\lambda$) radius of the pipes and the size of the velocity and density. Thus, for large $\mu$ our building blocks consist of a $1/\lambda$-periodic arrangement of very thin pipes of total volume fraction $1/\mu^{d-1}$ where the velocity and density are concentrated - see Proposition \ref{p:mikado} and Remark \ref{rmk:choice:ab} below.

\smallskip

We prove in details only Theorem \ref{thm:main}, in Sections \ref{s:technical}-\ref{s:proof:prop}. The proofs of Theorems \ref{thm:strong}, \ref{thm:diffusion}, \ref{thm:diffusion:higher} can be obtained from the one of Theorem \ref{thm:main} with minor changes. A sketch is provided in Section \ref{s:sketch}.

\subsection*{Acknowledgement} 
The authors would like to thank Gianluca Crippa for several very useful comments. This research was supported by the ERC Grant Agreement No. 724298.

\section{Technical tools}
\label{s:technical}

We start by fixing some notation:
\begin{itemize}
\item $\T^d = \R^d / \Z^d$ is the $d$-dimensional flat torus. 
\item For $p \in [1,\infty]$ we will always denote by $p'$ its dual exponent.
\item If $f(t,x)$ is a smooth function of $t \in [0,T]$ and $x \in \T^d$, we denote by 
	\begin{itemize}
	\item $\|f\|_{C^k}$ the sup norm of $f$ together with the sup norm of all its derivatives in time and space up to order $k$;
	\item $\|f(t, \cdot)\|_{C^k(\T^d)}$ the sup norm of $f$ together with the sup norm  of all its spatial derivatives up to order $k$ at fixed time $t$;
	\item $\|f(t,\cdot)\|_{L^p(\T^d)}$ the $L^p$ norm of $f$ in the spatial derivatives, at fixed time $t$. Since we will take always $L^p$ norms in the spatial variable (and never in the time variable), we will also use the shorter notation $\|f(t, \cdot)\|_{L^p} = \|f(t)\|_{L^p}$ to denote the $L^p$ norm of $f$ in the spatial variable.
	\end{itemize}
\item $C^\infty_0(\T^d)$ is the set of smooth functions on the torus with zero mean value.
\item $\N = \{0,1,2, \dots\}$.
\item We will use the notation $C(A_1, \dots, A_n)$ to denote a constant which depends only on the numbers $A_1, \dots, A_n$.
\end{itemize}

We now introduce three technical tools, namely an improved H\"older inequality, an antidivergence operator and a lemma about the mean value of fast oscillating functions. These tools will be frequently used in the following. For a function $g \in C^\infty(\T^d)$ and $\lambda \in \N$, we denote by $g_\lambda: \T^d \to \R$ the $1/\lambda$ periodic function defined by 
\begin{equation}
\label{eq:periodic}
g_\lambda (x) := g(\lambda x). 
\end{equation}
Notice that for every $k \in \N$ and $p \in [1, \infty]$
\begin{equation*}
\|D^k g_\lambda\|_{L^p(\T^d)} = \lambda^k \|D^k g\|_{L^p(\T^d)}.
\end{equation*}

\subsection{Improved H\"older inequality}

We start with the statement of the improved H\"older inequality, inspired by Lemma 3.7 in \cite{Buckmaster:2017wf}. 

\begin{lemma}
\label{l:improved:holder}
Let $\lambda \in \N$ and $f,g: \T^d \to \R$ be smooth functions. Then for every $p \in [1, \infty]$, 
\begin{equation}
\label{eq:improved:holder:1}
\bigg|\|fg_\lambda\|_{L^p} - \|f\|_{L^p} \|g\|_{L^p} \bigg| \leq \frac{C_p}{\lambda^{1/p}} \|f\|_{C^1} \|g\|_{L^p},
\end{equation}
where all the norms are taken on $\T^d$. 
In particular 
\begin{equation}
\label{eq:improved:holder:2}
\|fg_\lambda\|_{L^p} \leq \|f\|_{L^p} \|g\|_{L^p} + \frac{C_p}{\lambda^{1/p}} \|f\|_{C^1} \|g\|_{L^p}.
\end{equation}
\end{lemma}

\begin{proof}
Let us divide $\T^d$ into $\lambda^d$ small cubes $\{Q_j\}_j$ of edge $1/\lambda$. On each $Q_j$ we have
\begin{equation*}
\begin{split}
\int_{Q_j} 						& |f(x)|^p |g_\lambda(x)|^p \\
								& = \int_{Q_j} \bigg(|f(x)|^p - \fint_{Q_j} |f(y)|^p dy \bigg) |g_\lambda(x)|^p dx + \fint_{Q_j} |f(y)|^p dy \int_{Q_j} |g_\lambda(x)|^p dx \\
								& \text{(since $g_\lambda(x) = g(\lambda x)$)} \\
								& =  \int_{Q_j} \bigg(|f(x)|^p - \fint_{Q_j} |f(y)|^p dy \bigg) |g_\lambda(x)|^p dx + \frac{1}{\lambda^d} \fint_{Q_j} |f(y)|^p dy \int_{\T^d} |g(x)|^p dx \\
								& \text{(since $|Q_j| = 1/\lambda^d$)} \\
								& =  \int_{Q_j} \bigg(|f(x)|^p - \fint_{Q_j} |f(y)|^p dy \bigg) |g_\lambda(x)|^p dx + \int_{Q_j} |f(y)|^p dy \int_{\T^d} |g(x)|^p dx \
\end{split}
\end{equation*}
Summing over $j$ we get
\begin{equation*}
\begin{split}
\|fg_\lambda\|^p_{L^p(\T^d)} 	& = \|f\|^p_{L^p(\T^d)} \|g\|^p_{L^p(\T^d)} \\
								& \qquad + \sum_j \int_{Q_j} \bigg(|f(x)|^p - \fint_{Q_j} |f(y)|^p dy \bigg) |g_\lambda(x)|^p dx.
\end{split}
\end{equation*}
Let us now estimate the second term in the r.h.s. For $x,y \in Q_j$ it holds
\begin{equation*}
\begin{split}
\Big| |f(x)|^p - |f(y)|^p \Big| 		& \leq \frac{C_p}{\lambda} \|f\|^{p-1}_{C^0(\T^d)} \|\nabla f\|_{C^0(\T^d)} \leq \frac{C_p}{\lambda} \|f\|^p_{C^1(\T^d)}. 
\end{split}
\end{equation*}
Therefore
\begin{equation*}
\begin{split}
\sum_j \int_{Q_j} \bigg(|f(x)|^p - \fint_{Q_j} |f(y)|^p dy \bigg) |g_\lambda(x)|^p dx 
							& \leq \frac{C_p}{\lambda} \|f\|^p_{C^1(\T^d)} \sum_j  \int_{Q_j} |g_\lambda(x)|^p dx \\
							& = \frac{C_p}{\lambda} \|f\|^p_{C^1(\T^d)} \|g_\lambda\|^p_{L^p(\T^d)} \\
							& = \frac{C_p}{\lambda} \|f\|^p_{C^1(\T^d)} \|g\|^p_{L^p(\T^d)},
\end{split}
\end{equation*}
from which we get
\begin{equation*}
\bigg|\|fg_\lambda\|^p_{L^p} - \|f\|^p_{L^p} \|g\|^p_{L^p} \bigg| \leq \frac{C_p}{\lambda} \|f\|^p_{C^1} \|g\|^p_{L^p}.
\end{equation*}
Inequality \eqref{eq:improved:holder:1} is now obtained by taking the $1/p$ power in the last formula and using that for $A,B > 0$, $|A - B|^p \leq ||A|^p - |B|^p|$. Finally, the improved H\"older inequality \eqref{eq:improved:holder:2} is an immediate consequence of \eqref{eq:improved:holder:1}. 
\end{proof}

\subsection{Antidivergence operators}

For $f \in C^\infty_0(\T^d)$ there exists a unique $u \in C^\infty_0(\T^d)$ such that $\Delta u = f$. The operator $\Delta^{-1}: C^\infty_0(\T^d) \to C^\infty_0(\T^d)$ is thus well defined. We define the \emph{standard antidivergence operator} as $\nabla \Delta^{-1}: C^\infty_0(\T^d) \to C^\infty(\T^d; \R^d)$. It clearly satisfies $\div (\nabla \Delta^{-1} f) = f$. 

\begin{lemma}
\label{l:antidiv}
For every $k \in \N$ and $p \in [1, \infty]$, the \emph{standard} antidivergence operator satisfies the bounds \begin{equation}
\label{eq:stand:antidiv}
\big\|D^k (\nabla \Delta^{-1} g) \big\|_{L^p} \leq C_{k,p} \|D^k g\|_{L^p}.
\end{equation}
Moreover for every $\lambda \in \N$ it holds
\begin{equation}
\label{eq:stand:antidiv:per}
\big\|D^k (\nabla \Delta^{-1} g_\lambda)\big\|_{L^p} \leq C_{k,p}\lambda^{k-1} \|D^k g\|_{L^p}.
\end{equation}
\end{lemma}
\begin{proof}
For $p \in (1, \infty)$ from the Calderon-Zygmund inequality we get
\begin{equation}
\label{eq:cald:zyg}
\|D^k (\nabla \Delta^{-1} g)\|_{W^{1,p}(\T^d)} \leq C_{k,p} \|D^k g\|_{L^p(\T^d)},
\end{equation}
from which \eqref{eq:stand:antidiv} follows. For $p = \infty$, we use Sobolev embeddings to get
\begin{equation*}
\begin{split}
\|D^k (\nabla \Delta^{-1} g)\|_{L^\infty(\T^d)} 	& \leq C \|D^k (\nabla \Delta^{-1} g)\|_{W^{1, d+1}(\T^d)} \\
\text{(by \eqref{eq:cald:zyg} with $p=d+1$)}		& \leq C_{k, d+1} \|D^k g\|_{L^{d+1}(\T^d)} \\ 
													& \leq C_{k, \infty} \|D^k g\|_{L^{\infty}(\T^d)}.
\end{split}
\end{equation*}
For $p=1$ we use the dual characterization of $L^1$ norm. For every $f \in L^1$,
\begin{equation*}
\begin{split}
\|f\|_{L^1(\T^d)} 	& = \max \bigg\{ \int_{\T^d} f \varphi \ : \ \varphi \in L^\infty(\T^d), \ \|\varphi\|_{L^\infty(\T^d)} = 1 \bigg\} \\
					& = \sup \bigg\{ \int_{\T^d} f \varphi \ : \ \varphi \in C^\infty(\T^d), \ \|\varphi\|_{L^\infty(\T^d)} = 1 \bigg\}. 
\end{split}
\end{equation*}
Moreover, if $\fint_{\T^d} f = 0$, it also holds
\begin{equation*}
\|f\|_{L^1(\T^d)} = \sup \bigg\{ \int_{\T^d} f \varphi \ : \ \varphi \in C^\infty_0(\T^d), \ \|\varphi\|_{L^\infty(\T^d)} = 1 \bigg\}. 
\end{equation*}
Therefore (in the following formula $\partial^k$ denotes any partial derivative of order $k$):
\begin{equation*}
\begin{split}
\|\partial^k (\nabla \Delta^{-1} g)\|_{L^1(\T^d)}
& = \sup_{\substack{\varphi \in C^\infty_0(\T^d) \\ \|\varphi\|_{L^\infty} = 1}} \int_{\T^d} \partial^k (\nabla \Delta^{-1} g) \  \varphi  \ dx \\
& = \sup_{\substack{\varphi \in C^\infty_0(\T^d) \\ \|\varphi\|_{L^\infty} = 1}} \int_{\T^d} \partial^k g \ \nabla \Delta^{-1} \varphi \ dx \\
\text{(by H\"older)}
& \leq \sup_{\substack{\varphi \in C^\infty_0(\T^d) \\ \|\varphi\|_{L^\infty} = 1}}  \| \partial^k g \|_{L^1(\T^d)} \|\nabla \Delta^{-1} \varphi\|_{L^\infty(\T^d)} \\
\text{(using \eqref{eq:stand:antidiv} with $p=\infty$)}
& \leq C_{0, \infty} \| \partial^k g \|_{L^1(\T^d)} \sup_{\substack{\varphi \in C^\infty_0(\T^d) \\ \|\varphi\|_{L^\infty} = 1}} \| \varphi\|_{L^\infty(\T^d)}  \\
& \leq C_{0, \infty} \| \partial^k g \|_{L^1(\T^d)}, 
\end{split}
\end{equation*}
from with \eqref{eq:stand:antidiv} with $p=1$ follows. To prove \eqref{eq:stand:antidiv:per}, 
observe that 
\begin{equation*}
\nabla \Delta^{-1} g_\lambda(x) = \frac{1}{\lambda} (\nabla \Delta^{-1} g) (\lambda x).
\end{equation*}
Therefore 
\begin{equation*}
\begin{split}
\|D^k (\nabla \Delta^{-1} g_\lambda)\|_{L^p(\T^d)}
			& \leq \lambda^{k-1} \|(D^k (\nabla \Delta^{-1} g))(\lambda \ \cdot)\|_{L^p(\T^d)} \\
			& \leq \lambda^{k-1} \|D^k (\nabla \Delta^{-1} g)\|_{L^p(\T^d)} \\
\text{(by \eqref{eq:stand:antidiv})}
			& \leq C_{k,p}\lambda^{k-1} \|D^k g\|_{L^p(\T^d)},
\end{split}
\end{equation*}
thus proving \eqref{eq:stand:antidiv:per}.
\end{proof}

With the help of the standard antidivergence operator, we now define an \emph{improved} antidivergence operator, which lets us gain a factor $\lambda^{-1}$ when applied to functions of the form $f(x) g(\lambda x)$. 

\begin{lemma}
\label{l:improved:antidiv}
Let $\lambda \in \N$ and $f,g: \T^d \to \R$ be smooth functions with
\begin{equation*}
\fint_{\T^d} fg_\lambda = \fint_{\T^d} g = 0. 
\end{equation*}
Then there exists a smooth vector field $u : \T^d \to \R^d$ such that $\div u = f g_\lambda$ and for every   $k \in \N$ and $p \in [1,\infty]$, 
\begin{equation}
\label{eq:improved:antidiv}
\|D^k u\|_{L^p} \leq C_{k,p} \lambda^{k-1} \|f\|_{C^{k+1}} \|g\|_{W^{k,p}}.
\end{equation}
\end{lemma}
\noindent We will write
\begin{equation*}
u = \mathcal R (f g_\lambda).
\end{equation*}

\begin{remark}
The same result holds if $f, g$ are vector fields and we want to solve the equation $\div u = f \cdot g_\lambda$, where $\cdot$ denotes the scalar product. 
\end{remark}

\begin{proof}
Set
\begin{equation*}
u := f \nabla \Delta^{-1} g_\lambda - \nabla \Delta^{-1} \Big(\nabla f \cdot \nabla \Delta^{-1}g_\lambda  \Big).
\end{equation*}
It is immediate from the definition that $\div u = fg_\lambda$. We show that \eqref{eq:improved:antidiv} holds for $k=0,1$. The general case $k \in \N$ can be easily proven by induction. It holds (the constant $C_{0,p}$ can change its value from line to line)
\begin{equation*}
\begin{split}
\|u\|_{L^p} 	& \leq \|f\|_{C^0} \|\nabla \Delta^{-1}g_\lambda\|_{L^p} + \big\|\nabla \Delta^{-1} \big(\nabla f \cdot \nabla \Delta^{-1} g_\lambda \big) \big\|_{L^p} \\
\text{(by \eqref{eq:stand:antidiv})} 
				& \leq \|f\|_{C^0} \|\nabla \Delta^{-1}g_\lambda\|_{L^p} + C_{0,p} \|\nabla f\|_{C^0} \|\nabla \Delta^{-1} g_\lambda \|_{L^p} \\
\text{(by \eqref{eq:stand:antidiv:per})} 
				& \leq \frac{C_{0,p}}{\lambda} \|f\|_{C^1} \|g\|_{L^p},
\end{split}
\end{equation*}
so that \eqref{eq:improved:antidiv} holds for $k=0$. For $k=1$ we compute
\begin{equation*}
\partial_j u = \partial_j f \nabla \Delta^{-1} g_\lambda + f \partial_j \nabla \Delta^{-1}  g_\lambda - \nabla \Delta^{-1} \Big( \nabla \partial_j f \cdot \nabla \Delta^{-1} g_\lambda + \nabla f \cdot \partial_j \nabla \Delta^{-1}  g_\lambda \Big).
\end{equation*}
Therefore, using again \eqref{eq:stand:antidiv} and \eqref{eq:stand:antidiv:per}, (the constant $C_{1,p}$ can change its value from line to line)
\begin{equation*}
\begin{split}
\|\partial_j u\|_{L^p} 	& \leq C_{1,p} \bigg[ \|\partial_j f\|_{C^0} \|\nabla \Delta^{-1} g_\lambda\|_{L^p} + \|f\|_{C^0} \|\partial_j \nabla \Delta^{-1}  g_\lambda \|_{L^p} \\
							& \qquad +  \|\nabla \partial_j f\|_{C^0} \|\nabla \Delta^{-1} g_\lambda\|_{L^p} +  \|\nabla f\|_{C^0} \|\partial_j\nabla \Delta^{-1}  g_\lambda \|_{L^p} \bigg]\\
							& \leq C_{1,p} \bigg[ \frac{1}{\lambda} \|f\|_{C^1} \|g\|_{L^p} + \|f\|_{C^0} \|\partial_j g\|_{L^p} \\
							& \qquad \qquad + \frac{1}{\lambda} \|f\|_{C^2} \|g\|_{L^p} + \|f\|_{C^1} \|\partial_j g\|_{L^p} \bigg] \\
							& \leq C_{1,p} \bigg[ \|f\|_{C^1} \|\partial_j g\|_{L^p} + \frac{1}{\lambda} \|f\|_{C^2} \|g\|_{L^p} \bigg] \\
							& \leq C_{1,p} \|f\|_{C^2} \|g\|_{W^{1,p}}. \qedhere
\end{split}
\end{equation*}
\end{proof}

\begin{remark}
\label{rmk:time:as:param}
Assume  $f$ and $g$ are smooth function of $(t,x)$, $t \in [0,T]$, $x \in \T^d$. If at each time $t$ they satisfy (in the space variable) the assumptions of Lemma \ref{l:improved:antidiv}, then we can apply $\mathcal R$ at each time and define
\begin{equation*}
u(t, \cdot) := \mathcal R \Big( f(t, \cdot) g_\lambda(t, \cdot) \Big),
\end{equation*}
where $g_\lambda(t,x) = g(t, \lambda x)$. It follows from the definition of $\mathcal R$ that $u$ is a smooth function of $(t,x)$. 
\end{remark}

\subsection{Mean value and fast oscillations}

\begin{lemma}
\label{l:mean:value}
Let $\lambda \in \N$ and $f,g: \T^d \to \R$ be smooth functions with
\begin{equation*}
\fint_{\T^d} g(x) dx = 0. 
\end{equation*} 
Then
\begin{equation*}
\bigg| \fint_{\T^d} f(x) g(\lambda x) dx \bigg| \leq  \frac{\sqrt{d} \|f\|_{C^1}\|g\|_{L^1}}{\lambda}.  
\end{equation*}
\end{lemma}
\begin{proof}
We divide $\T^d$ into small cubes $\{Q_j\}$ of edge $1/\lambda$. For each $Q_j$, choose a point $x_j \in Q_j$. We have
\begin{equation*}
\begin{split}
\bigg|\int_{\T^d} f(x) g(\lambda x) dx \bigg|
& = \bigg| \sum_j \int_{Q_j} f(x) g(\lambda x) dx \bigg|\\
& = \bigg|\sum_j \int_{Q_j} \big[ f(x) - f(x_j) \big] g(\lambda x) dx \bigg| \\
& \leq \sum_j \int_{Q_j} \big| f(x) - f(x_j) \big| |g(\lambda x)| dx  \\
& \leq \frac{\sqrt{d} \|f\|_{C^1}\|g\|_{L^1(\T^d)}}{\lambda}.  \qedhere
\end{split}
\end{equation*}
\end{proof}

\section{Statement of the main proposition and proof of Theorem \ref{thm:main}}

We assume without loss of generality that $T=1$ and $\T^d$ is the periodic extension of the unit cube $[0,1]^d$. 
The following proposition contains the key facts used to prove Theorem \ref{thm:main}. Let us first introduce the \emph{continuity-defect} equation:
\begin{equation}
\label{eq:cont:reyn}
\left\{
\begin{split}
\partial_t \rho + \div(\rho u) 	& = - \div R, \\
\div u 								& = 0.
\end{split}
\right.
\end{equation}
We will call $R$ the \emph{defect field}. For $\sigma>0$ set $I_\sigma = (\sigma, 1 - \sigma)$. Recall that we are assuming $p \in (1, \infty)$.

\begin{proposition}
\label{p:main}
There exists a constant $M>0$ such that the following holds. Let $\eta, \delta, \sigma > 0$ and let $(\rho_0, u_0, R_0)$ be a smooth solution of the continuity-defect equation \eqref{eq:cont:reyn}. Then there exists another smooth solution $(\rho_1, u_1, R_1)$ of \eqref{eq:cont:reyn} such that
\begin{subequations}
\begin{align}
\|\rho_1(t) - \rho_0(t)\|_{L^p(\T^d)}		& \leq 
\begin{cases}
M \eta \|R_0(t)\|^{1/p}_{L^1(\T^d)}, 		& t \in I_{\sigma/2},  \\
0, 											& t \in [0,1] \setminus I_{\sigma/2},
\end{cases} \label{eq:dist:rho:stat} \\
\|u_1(t) - u_0(t)\|_{L^{p'}(\T^d)}			& \leq 
\begin{cases}
M \eta^{-1} \|R_0(t)\|^{1/p'}_{L^1(\T^d)}, 		& t \in I_{\sigma/2}, \\
0, 													& t \in [0,1] \setminus I_{\sigma/2}, 
\end{cases} \label{eq:dist:u:1:stat} \\
\|u_1(t) - u_0(t)\|_{W^{1,\tilde p}(\T^d)} 	& \leq \delta, \label{eq:dist:u:2:stat} \\[.1em]
\|R_1(t)\|_{L^1(\T^d)} 						& \leq
\begin{cases}
\delta,										& t \in I_\sigma, \\
\|R_0(t)\|_{L^1(\T^d)} + \delta,			& t \in I_{\sigma/2} \setminus I_\sigma, \\
\|R_0(t)\|_{L^1(\T^d)},					& t \in [0,1] \setminus I_{\sigma/2}.
\end{cases} \label{eq:reyn:stat}
\end{align}
\end{subequations}
\end{proposition}

\begin{proof}[Proof of Theorem \ref{thm:main} assuming Proposition \ref{p:main}]
Let $M$ be the constant in Proposition \ref{p:main}. Let $\tilde \e>0$ and $\eta >0$ (their precise value will be fixed later, with $\eta$ depending on $\tilde \e$). Let $\sigma_q = \delta_q := 2^{-q}$ and $I_q := I_{\sigma_q} = (\sigma_q, 1 - \sigma_q)$. 

We construct a sequence $(\rho_q, u_q, R_q)$ of solutions to \eqref{eq:cont:reyn} as follows. Let $\phi_0, \phi, \phi_1: [0,1] \to \R$ three smooth functions such that 
\begin{equation*}
\phi_0(t) + \phi(t) + \phi_1(t) = 1 \quad \text{ for every } t \in [0,1],
\end{equation*}
and
\begin{equation*}
\begin{split}
\phi_0(t) = 1 		& \text{ on } [0,\tilde \e], \\
\phi(t) = 1 		& \text{ on } [2\tilde \e, 1 - 2 \tilde \e], \\
\phi_1(t) = 1 		& \text{ on } [1 - \tilde \e, 1].
\end{split}
\end{equation*}
Set  
\begin{equation}
\label{eq:starting:data}
\begin{split}
\rho_0(t) 	& := \phi_0(t) \bar \rho(0) + \phi(t) \bar \rho(t) + \phi_1(t) \bar \rho(1), \\
u_0(t)	 	& := 0,  \\
R_0(t) 		& := - \nabla \Delta^{-1} \big( \partial_t \rho_0(t) + \div(\rho_0(t) u_0(t)) \big) = - \nabla \Delta^{-1} \big( \partial_t \rho_0(t) \big),
\end{split}
\end{equation}
where the antidivergence is taken with respect to the spatial variable.

Assume now that $(\rho_q, u_q, R_q)$ is defined. Let $(\rho_{q+1}, u_{q+1}, R_{q+1})$ be the solution to the continuity-defect equation, which is obtained by applying Proposition \ref{p:main} to $(\rho_q, u_q, R_q)$, $\eta$, 
\begin{equation*}
\delta = \delta_{q+2}, \quad  \sigma = \sigma_{q+1} \quad \text{(and thus  $\sigma/2 = \sigma_{q+2}$)}.
\end{equation*}

\begin{lemma}
\label{l:inductive:est}
The following inductive estimates are satisfied:
\begin{subequations}
\refstepcounter{equation}\label{eq:dist:rho}
\refstepcounter{equation}\label{eq:dist:u:1}
\refstepcounter{equation}\label{eq:dist:u:2}
\refstepcounter{equation}\label{eq:reyn}
\begin{align}
\tag*{\mytag{eq:dist:rho}{q}}
\|\rho_q(t) - \rho_{q-1}(t)\|_{L^p}								& \leq
\begin{cases}
M \eta \delta_{q}^{1/p}, 		& t \in I_{q-1}, \\
M \eta [ \|R_0(t)\|_{L^1} + \delta_{q}]^{1/p}, 	& t \in I_q \setminus I_{q-1}, \\
M \eta \|R_0(t)\|^{1/p}_{L^1}, 							& t \in I_{q+1} \setminus I_q, \\
0, 								& t \in [0,1] \setminus I_{q+1},
\end{cases}  \\
\tag*{\mytag{eq:dist:u:1}{q}}
\|u_{q}(t) - u_{q-1}(t)\|_{L^{p'}} 								& \leq 
\begin{cases}
M \eta^{-1} \delta_{q}^{1/p'}, 		& t \in I_{q-1}, \\
M \eta^{-1} [ \|R_0(t)\|_{L^1} + \delta_{q}]^{1/p'}, 	& t \in I_q \setminus I_{q-1}, \\
M \eta^{-1} \|R_0(t)\|^{1/p'}_{L^1}, 							& t \in I_{q+1} \setminus I_q, \\
0, 								& t \in [0,1] \setminus I_{q+1},
\end{cases}  \\
\tag*{\mytag{eq:dist:u:2}{q}}
\|u_{q}(t) - u_{q-1}(t)\|_{W^{1,\tilde p}} 						& \leq \delta_{q+1},   \\
\|R_{q}(t)\|_{L^1}												& \leq 
\begin{cases}
\delta_{q+1}, 							& t \in I_{q}, \\
\|R_0(t)\|_{L^1} + \delta_{q+1},		& t \in I_{q+1} \setminus  I_{q}, \\
\|R_0(t)\|_{L^1}, 						& t \in [0,1] \setminus  I_{q+1}. \\
\end{cases}	  
\tag*{\mytag{eq:reyn}{q}}
\end{align}
\end{subequations}
\end{lemma}
\begin{proof}
For $q=0$, \mytag{eq:dist:rho}{q}-\mytag{eq:dist:u:2}{q} do not apply, whereas \mytag{eq:reyn}{q} is trivially satisfied, since $I_0 = \emptyset$. Assume now that \mytag{eq:dist:rho}{q}-\mytag{eq:reyn}{q} hold and let us prove \mytag{eq:dist:rho}{q+1}-\mytag{eq:reyn}{q+1}. From \eqref{eq:dist:rho:stat} we get
\begin{equation*}
\|\rho_{q+1}(t) - \rho_q(t)\|_{L^p} \leq 
\begin{cases}
M \eta \|R_q(t)\|_{L^1}^{1/p}, 	& t \in I_{q+2}, \\
0, 									& t \in [0,1] \setminus I_{q+2}. 
\end{cases}
\end{equation*}
Therefore, using the inductive assumption \mytag{eq:reyn}{q}, we get:
\begin{itemize}
\item if $t \in I_q$,
\begin{equation*}
\|\rho_{q+1}(t) - \rho_q(t)\|_{L^p} \leq M \eta \|R_q(t)\|_{L^1}^{1/p}, \leq M \eta \delta_{q+1}^{1/p};
\end{equation*}
\item if $t \in I_{q+1} \setminus I_q$, 
\begin{equation*}
\|\rho_{q+1}(t) - \rho_q(t)\|_{L^p} \leq M \eta \|R_q(t)\|_{L^1}^{1/p} \leq M \eta \Big[|R_0(t)\|_{L^1} + \delta_{q+1} \Big]^{1/p};
\end{equation*}
\item if $t \in I_{q+2} \setminus I_{q+1}$,
\begin{equation*}
\|\rho_{q+1}(t) - \rho_q(t)\|_{L^p} \leq M \eta \|R_q(t)\|_{L^1}^{1/p} 
\leq M \eta \|R_0(t)\|_{L^1}^{1/p},
\end{equation*}
\end{itemize}
and thus \mytag{eq:dist:rho}{q+1} holds. Estimate \mytag{eq:dist:u:1}{q+1} can be proven similarly. Estimate \mytag{eq:dist:u:2}{q+1} is an immediate consequence of \eqref{eq:dist:u:2:stat}. Finally, from \eqref{eq:reyn:stat}, we get
\begin{equation*}
\|R_{q+1}(t)\|_{L^1} 						 \leq
\begin{cases}
\delta_{q+2},											& t \in I_{q+1}, \\
\|R_q(t)\|_{L^1} + \delta_{q+2},				& t \in I_{q+2} \setminus I_{q+1}, \\
\|R_q(t)\|_{L^1},								& t \in [0,1] \setminus I_{q+2}.
\end{cases}
\end{equation*}
Therefore, using the inductive assumption \mytag{eq:reyn}{q}, we get:
\begin{itemize}
\item if $t \in I_{q+2} \setminus I_{q+1}$,
\begin{equation*}
\|R_{q+1}(t)\|_{L^1} \leq \|R_q(t)\|_{L^1} + \delta_{q+2} \leq \|R_0(t)\|_{L^1} + \delta_{q+2};
\end{equation*}
\item if $t \in [0,1] \setminus I_{q+2}$, 
\begin{equation*}
\|R_{q+1}(t)\|_{L^1} \leq \|R_q(t)\|_{L^1} \leq \|R_0(t)\|_{L^1},
\end{equation*}
\end{itemize}
from which \mytag{eq:reyn}{q+1} follows. 
\end{proof}

\bigskip

\bigskip
It is now an immediate consequence of the previous lemma that there exists 
\begin{equation}
\label{eq:regularity}
\rho \in C((0,1); L^p(\T^d)), \qquad u \in C((0,1); W^{1, \tilde p}(\T^d)) \cap C((0,1); L^{p'}(\T^d))
\end{equation}
such that for every compact subset $K \subseteq (0,1)$
\begin{equation*}
\begin{split}
\max_{t \in K} \|\rho_q(t) - \rho(t)\|_{L^p} 			& \to 0 \\
\max_{t \in K} \|u_q(t) - u(t)\|_{L^{p'}} 			& \to 0 \\
\max_{t \in K} \|u_q(t) - u(t)\|_{W^{1, \tilde p}} 	& \to 0 \\
\max_{t \in K} \|R_q(t)\|_{L^1} 						& \to 0,
\end{split}
\end{equation*}
as $q \to \infty$, from which it follows that $\rho, u$ solves \eqref{e:incompressible}-\eqref{eq:continuity} (or \eqref{e:transport}-\eqref{e:incompressible}) in the sense of distributions. This proves part (b) of the statement.

We need now the following estimate. Let $t \in (0,1)$ and let $q^* = q^*(t) \in \N$ so that $t \in I_{q^*} \setminus I_{q^* -1}$. By the inductive estimate \eqref{eq:dist:rho},
\begin{equation}
\label{eq:distance:initial:final}
\begin{split}
\|\rho(t) - \rho_0(t)\|_{L^p}		& \leq \sum_{q=1}^\infty \|\rho_q(t) - \rho_{q-1}(t)\|_{L^p} \\
									& \leq \|\rho_{q^*-1}(t) - \rho_{q^*-2}(t)\|_{L^p} + \|\rho_{q^*}(t) - \rho_{q^*-1}(t)\|_{L^p} \\
									& \quad + \sum_{q=q^*+1}^\infty \|\rho_q(t) - \rho_{q-1}(t)\|_{L^p} \\
									& \leq M \eta \Bigg[ \|R_0(t)\|^{1/p}_{L^1} + \Big( \|R_0(t)\|_{L^1} + \delta_{q^*}\Big) ^{1/p} +  \sum_{q=q^*+1}^\infty \delta_q^{1/p} \Bigg].
\end{split}
\end{equation}

Let us now prove that $\|\rho(t) - \bar \rho(0)\|_{L^p} \to 0$ as $t \to 0$. Observe that, for $t < \tilde \e$, $\rho_0(t) = \bar \rho(0)$ and $R_0(t) = 0$. Hence, if $t < \tilde \e$, 
\begin{equation*}
\begin{split}
\|\rho(t) - \bar \rho(0)\|_{L^p}	& =  \|\rho(t) - \rho_0(t)\|_{L^p} \\
\text{(by \eqref{eq:distance:initial:final})}
									& \leq M \eta \Bigg[ \|R_0(t)\|^{1/p}_{L^1} + \Big( \|R_0(t)\|_{L^1} + \delta_{q^*}\Big) ^{1/p} +  \sum_{q=q^*+1}^\infty \delta_q^{1/p} \Bigg] \\
									& \leq M \eta \sum_{q=q^*}^\infty \delta_q^{1/p},
\end{split}
\end{equation*}
and the conclusion follows observing that $q^* = q^*(t) \to \infty$ as $t \to 0$. In a similar way the limit $\|\rho(t) - \bar \rho(1)\|_{L^p} \to 0$ as $t \to 1$ can be shown. This completes the proof of parts (a) and (c) of the statement. 

Let us now prove part (d). We first observe that, for the $\e$ given in the statement of the theorem, we can choose $\tilde \e$ small enough, so that for every $t \in [0,1]$,
\begin{equation}
\| \rho_0(t) - \bar \rho(t)\|_{L^p} \leq \frac{\e}{2}.
\end{equation}
Indeed, if $t \in [2 \tilde \e, 1 - 2\tilde \e]$, then $\rho_0(t) = \bar \rho(t)$. If $t \in [0, 2\tilde \e] \cup [1 - 2 \tilde \e, 1]$, then 
\begin{equation*}
\begin{split}
\|\rho_0(t) - \bar \rho(t)\|_{L^p} 	& \leq |\phi_0(t)| \|\bar \rho(0) - \bar \rho(t)\|_{L^p} + |\phi_1(t)| \|\bar \rho(1) - \bar \rho(t)\|_{L^p}  \leq \frac{\e}{2},
\end{split}
\end{equation*}
where the last inequality follows, by choosing $\tilde \e$ sufficiently small.   Therefore, for every $t \in [0,1]$,
\begin{equation*}
\begin{split}
\|\rho(t) - \bar \rho(t)\|_{L^p} 			& \leq \|\rho(t) - \rho_0(t)\|_{L^p} + \|\rho_0(t) - \bar \rho(t)\|_{L^p} \\
\text{(by \eqref{eq:distance:initial:final})}
											& \leq M \eta \Bigg[ \|R_0(t)\|^{1/p}_{L^1} + \Big( \|R_0(t)\|_{L^1} + \delta_{q^*}\Big) ^{1/p} +  \sum_{q=q^*+1}^\infty \delta_q^{1/p} \Bigg] + \frac{\e}{2} \\
											& \leq M \eta  \max_{t \in [0,1]} \Bigg[\|R_0(t)\|^{1/p}_{L^1} + \Big( \|R_0(t)\|_{L^1} + 1 \Big) ^{1/p} +  \sum_{q=1}^\infty \delta_q^{1/p} \Bigg] + \frac{\e}{2} \\
											& \leq \e,
\end{split}
\end{equation*}
if $\eta$ is chosen small enough (depending on $R_0$ and thus on $\tilde \e$). This proves part (d) of the statement, thus concluding the proof of the theorem. \qedhere

\end{proof}

\section{The perturbations}
\label{s:perturbations}

In this and the next two sections we prove Proposition \ref{p:main}. In particular in this section we fix the constant $M$ in the statement of the proposition, we define the functions $\rho_1$ and $u_1$ and we prove some estimates on them. In Section \ref{s:reynolds} we define $R_1$ and we prove some estimates on it. In Section \ref{s:proof:prop} we conclude the proof of Proposition \ref{p:main}, by proving estimates \eqref{eq:dist:rho:stat}-\eqref{eq:reyn:stat}.

\subsection{Mikado fields and Mikado densities}

The first step towards the definition of $\rho_1, u_1$ is the construction of \emph{Mikado fields} and \emph{Mikado densities}.

We start by fixing a function $\Phi \in C^\infty_c(\R^{d-1})$ such that 
\begin{equation*}
\supp \Phi \subseteq (0,1)^{d-1}, \quad \int_{\R^{d-1}} \Phi = 0, \quad  \int_{\R^{d-1}} \Phi^2 = 1.
\end{equation*}
Let $\Phi_\mu(x) := \Phi(\mu x)$ for $\mu > 0$. Let $a \in \R$. For every $k \in \N$, it holds
\begin{equation}
\label{eq:rescale:phi}
\begin{split}
\|D^k (\mu^a \Phi_\mu)\|_{L^r(\R^{d-1})}  	& = \mu^{a+k- (d-1)/r} \|D^k\Phi\|_{L^r(\R^{d-1})}, \\
\end{split}
\end{equation}

\begin{proposition}
\label{p:mikado}
Let $a, b \in \R$ with 
\begin{equation}
\label{eq:exponent}
a+b = d-1.
\end{equation}
For every $\mu>2d$ and $j=1, \dots, d$ there exist a \emph{Mikado density} $\Theta_{\mu}^{j} : \T^d \to \R$ and a \emph{Mikado field} $W_\mu^j :\T^d \to \R^d$ with the following properties.
\begin{enumerate}[(a)]

\item It holds
\begin{equation}
\label{eq:mikado:eq}
\begin{cases}
\div W_\mu^j & = 0, \\
\div (\Theta_\mu^j W_\mu^j) & = 0, \\
\fint_{\T^d} \Theta_{\mu}^j = \fint_{\T^d} W_\mu^j & = 0, \\
\fint_{\T^d} \Theta_\mu^j W_\mu^j & = e_j, \\
\end{cases}
\end{equation}
where $\{e_j\}_{j=1,\dots,d}$ is the standard basis in $\R^d$.

\item For every $k \in \N$ and  $r \in [1,\infty]$
\begin{equation}
\label{eq:mikado:est}
\begin{split}
\|D^k \Theta_\mu^j\|_{L^r(\T^d)} 			& \leq \|\Phi\|_{L^r(\R^{d-1})} \mu^{a + k - (d-1) /r}, \\
\|D^k W_\mu^j\|_{L^{r}(\T^d)}				& \leq \|\Phi\|_{L^r(\R^{d-1})} \mu^{b + k - (d-1)/r}, \\
\end{split}
\end{equation}

\item For $j \neq k$, $\supp \Theta_\mu^j = \supp W_\mu^j$ and $\supp \Theta_\mu^j \cap \supp W_{\mu}^k = \emptyset$.

\end{enumerate}
\end{proposition}


\begin{remark}
\label{rmk:choice:ab}
In particular notice that if we choose
\begin{equation*}
a = \frac{d-1}{p}, \qquad b = \frac{d-1}{p'}
\end{equation*}
and we define the constant $M$ in the statement of Proposition \ref{p:main} as
\begin{equation}
\label{eq:M}
M := 2d \max \Big\{ \|\Phi\|_{L^\infty(\R^{d-1})},  \|\Phi\|^2_{L^\infty(\R^{d-1})}, \|\nabla \Phi\|_{L^\infty(\R^{d-1})}\Big\}.
\end{equation}
then the following estimates holds:
\begin{equation}
\label{eq:mikado:est:1}
\begin{split}
\sum_{j=1}^d \|\Theta_\mu^j\|_{L^p(\T^d)}, \  \sum_{j=1}^d \|W_\mu^j\|_{L^{p'}(\T^d)}, \ \sum_{j=1}^d	\|\Theta_\mu^j W_\mu^j\|_{L^1(\T^d)}		& \leq \frac{M}{2}, \\
\end{split}
\end{equation}
and
\begin{equation}
\label{eq:mikado:est:2}
\begin{split}
\|\Theta_\mu^j\|_{L^1(\T^d)}, \ \|W_\mu^j\|_{L^{1}(\T^d)}, \  \|W_\mu^j\|_{W^{1, \tilde p}} 	
\leq M \mu^{-\gamma},
\end{split}
\end{equation}
%
%
where

\begin{equation*}
\gamma = \min \big\{ \gamma_1, \gamma_2, \gamma_3 \big\} > 0
\end{equation*}
and
\begin{equation*}
\begin{split}
\gamma_1 & := (d-1) \bigg(1 - \frac{1}{p} \bigg) > 0, \\
\gamma_2 & := (d-1) \bigg(1 - \frac{1}{p'} \bigg) > 0, \\
\gamma_3 & := - 1 - (d-1) \bigg[\frac{1}{p'} - \frac{1}{\tilde p} \bigg] = (d-1) \bigg[\frac{1}{p} + \frac{1}{\tilde p}   - \bigg( 1 + \frac{1}{d-1} \bigg) \bigg] > 0.
\end{split}
\end{equation*}
Notice that $\gamma_3 >0$ by \eqref{eq:p:ineq}. 
\end{remark}

\begin{proof}[Proof of Proposition \ref{p:mikado}]

\textit{Step 1.} For each $j=1,\dots, d$, we define the (non-periodic) Mikado density $\tilde \Theta_{\mu}^j : \R^d \to \R$
\begin{subequations}
\label{eq:non:periodic}
\begin{equation}
\label{eq:non:periodic:dens}
\tilde \Theta_\mu^j(x_1, \dots, x_n) := \mu^a \Phi_\mu(x_1, \dots, x_{j-1}, x_{j+1}, \dots, x_d)
\end{equation}
and the (non-periodic) Mikado field $\tilde W_\mu^j : \R^d \to \R^d$
\begin{equation}
\label{eq:non:periodic:field}
\tilde W_\mu^j (x_1, \dots, x_n) := \mu^b \Phi_\mu(x_1, \dots, x_{j-1}, x_{j+1}, \dots, x_d) e_j.
\end{equation}
\end{subequations}
Notice that for the non-periodic Mikado densities 
\begin{equation}
\label{eq:mikado:notper:eqn}
\begin{cases}
\div \tilde W_\mu^j & = 0, \\
\div (\tilde \Theta_\mu^j \tilde W_\mu^j) & = 0, \\
\fint_{(0,1)^d} \tilde \Theta_\mu^j = \fint_{(0,1)^d} \tilde W_\mu^j 		& = 0, \\
\fint_{(0,1)^d} \tilde \Theta_\mu^j \tilde W_\mu^j 									& = e_j, \\
\end{cases}
\end{equation}
where the last equality follows from \eqref{eq:exponent}. Moreover, from \eqref{eq:rescale:phi} we get
\begin{equation}
\label{eq:mikado:notper:est}
\begin{split}
\|D^k \tilde \Theta_\mu^j\|_{L^r((0,1)^d)} 		& = \mu^{a + k - (d-1)/r}\|D^k\Phi\|_{L^r(\R^{d-1})} \\
\|D^k \tilde W_\mu^j\|_{L^{r}((0,1)^d)} 			& = \mu^{b + k - (d-1)/r}\|D^k\Phi\|_{L^{r}(\R^{d-1})} \\
\end{split}
\end{equation}

\smallskip

\textit{Step 2.} We define $\Theta_\mu^j : \T^d \to \R$ and $W_\mu^j :\T^d \to \R^d$ as the $1$-periodic extension of $\tilde \Theta_\mu^j$, $\tilde W_\mu^j$ respectively.  Such periodic extensions are well defined, since $\supp \Phi \subseteq (0,1)^{d-1}$ and $\tilde \Theta_\mu^j$, $\tilde W_\mu^j$ does not depend on the $j$-th coordinate. Equations \eqref{eq:mikado:eq} and estimates \eqref{eq:mikado:est} come from the corresponding equations \eqref{eq:mikado:notper:eqn} and estimates \eqref{eq:mikado:notper:est} for the non-periodic Mikado densities and fields. 

\textit{Step 3.} Finally notice that conditions (c) in the statement are not verified by $\Theta_\mu^j$ and $W_\mu^j$ defined in Step 2. However we can achieve (c), using that $\mu > 2d$ and  redefining $\Theta_\mu^j, W_\mu^j$ after a suitable translation of the independent variable $x \in \T^d$ for each $j=1,\dots, d$.
\end{proof}

\subsection{Definition of the perturbations}
\label{ss:def:pert}

We are now in a position to define $\rho_1$, $u_1$. The constant $M$ has already been fixed in \eqref{eq:M}. Let thus $\eta, \delta, \sigma>0$ and $(\rho_0, u_0, R_0)$ be a smooth solution to the continuity-defect equation \eqref{eq:cont:reyn}.

Let
\begin{equation*}
\begin{split}
\lambda \in \N & \text{ ``oscillation''}, \\
\mu > 2d & \text{ ``concentration''} 
\end{split}
\end{equation*}
be two constant, which will be fixed in Section \ref{s:proof:prop}.
 Let $\psi \in C^\infty_c((0,1))$ such that $\psi \equiv 0$ on $[0,\sigma/2] \cup [1-\sigma/2, 1]$, $\psi \equiv 1$ on $[\sigma, 1 - \sigma]$ and $|\psi| \leq 1$. We denote by $R_{0,j}$ the components of $R_0$, i.e.
\begin{equation*}
R_0(t,x) := \sum_{j=1}^d R_{0,j}(t,x) e_j. 
\end{equation*}
For $j=1,\dots, d$, let $\chi_j \in C^\infty ([0,1] \times \T^d)$ be such that
\begin{equation*}
\chi_j(t,x) =
\begin{cases}
0, & \text{if } |R_{0,j}(t,x)| \leq \delta/(4d), \\
1, & \text{if } |R_{0,j}(t,x)| \geq \delta/(2d),
\end{cases}
\end{equation*}
%
and $|\chi_j| \leq 1$. 

\bigskip

We set
\begin{equation*}
\rho_1 := \rho_0 + \vartheta + \vartheta_c , \qquad u_1 :=  u_0 + w + w_c,
\end{equation*}
where $\vartheta, \vartheta_c, w, w_c$ are defined as follows. First of all, let $\Theta_\mu^j$, $W_\mu^j$, $j=1,\dots, d$, be the Mikado densities and flows provided by Proposition \ref{p:mikado}, with $a,b$ chosen as in Remark \ref{rmk:choice:ab}. We set
\begin{equation}
\label{eq:perturbation}
\begin{split}
\vartheta (t,x) 	& := \eta \sum_{j=1}^d \psi(t) \chi_j(t,x) \sign \big( R_{0,j}(t,x) \big) \big|R_{0,j}(t,x)\big|^{1/p} \Theta_\mu^j(\lambda x), \\
\vartheta_c(t)		& := - \fint_{\T^d} \vartheta(t,x) dx, \\
w(t,x) 				& := \eta^{-1}\sum_{j=1}^d \psi(t) \chi_j(t,x) \big|R_{0,j}(t,x)\big|^{1/p'} W_\mu^j(\lambda x).
\end{split}
\end{equation}
We will also use the shorter notation
\begin{equation*}
\begin{split}
\vartheta(t) 	& = \eta \sum_{j=1}^d \psi(t) \chi_j(t) \sign(R_{0,j}(t)) |R_{0,j}(t)|^{1/p} \big( \Theta_\mu^j \big)_\lambda, \\
w(t)			& = \eta^{-1} \sum_{j=1}^d \psi(t) \chi_j(t) |R_{0,j}(t)|^{1/p'} \big( W_\mu^j \big)_\lambda,
\end{split}
\end{equation*}
where, coherent with \eqref{eq:periodic}, 
\begin{equation*}
\big( \Theta_\mu^j \big)_\lambda(x) = \Theta_\mu^j (\lambda x), \qquad \big( W_\mu^j \big)_\lambda(x) = W_\mu^j (\lambda x).
\end{equation*}
Notice that $\vartheta_0$ and $w$ are smooth functions, thanks to the cutoffs $\chi_j$. Notice also that $\vartheta + \vartheta_c$ has zero mean value in $\T^d$ at each time $t$. To define $w_c$, notice first that
\begin{equation*}
- \div  w(t)  = - \eta^{-1} \sum_{j=1}^d \nabla \Big( \psi(t) \chi_j(t) \big|R_{0,j}(t)\big|^{1/p} \Big) \cdot \big(W_\mu^j \big)_\lambda
\end{equation*}
is  sum of terms of the form $f \cdot g_\lambda$, each term has zero mean value (being a divergence) and the fast oscillatory term $W_\mu^j$ has zero mean value as well. We can therefore apply Lemma \ref{l:improved:antidiv} and define
\begin{equation}
\label{eq:wc}
w_c(t) := - \eta^{-1} \sum_{j=1}^d \mathcal R \bigg( \nabla \Big( \psi(t) \chi_j(t) \big|R_{0,j}(t)\big|^{1/p} \Big) \cdot \big(W_\mu^j \big)_\lambda \bigg). 
\end{equation}
Then $\div (w + w_c) = 0$ and thus
\begin{equation*}
\div u_1 = \div u_ 0 + \div(w + w_c) = 0.
\end{equation*}
Moreover, by Remark \ref{rmk:time:as:param}, $w_c$ is smooth in $(t,x)$.

\subsection{Estimates on the perturbation}\label{ss:estimates-perturbation}

In this section we provide some estimates on $\vartheta$, $\vartheta_c$, $w$, $w_c$.
%
%

\begin{lemma}[$L^p$-norm of $\vartheta$]
\label{l:lp:vartheta}
For every time $t \in [0,1]$,
\begin{equation*}
\|\vartheta(t)\|_{L^p(\T^d)} \leq \frac{M}{2} \eta \|R_0(t)\|^{1/p}_{L^1(\T^d)} +  \frac{C(\eta, \delta, \|R_0(t)\|_{C^1(\T^d)})}{\lambda^{1/p}}.
\end{equation*}
%
\end{lemma}
\begin{proof}
The perturbation $\vartheta$ is the sum of functions of the form $f g_\lambda$. Therefore we can apply the improved H\"older inequality, Lemma \ref{l:improved:holder}, to get
\begin{equation*}
\begin{split}
\|\vartheta(t)\|_{L^p} 
& \leq \eta \sum_{j=1}^d \bigg\| \psi(t)\chi_j(t) \sign \big( R_{0,j}(t) \big) \big|R_{0,j}(t)\big|^{1/p} \bigg\|_{L^p} \|\Theta_\mu^j\|_{L^p} \\
& \qquad +  \frac{C_p}{\lambda^{1/p}} \Big\| \psi(t) \chi_j(t) \sign ( R_{0,j}(t)) \big|R_{0,j}(t)\big|^{1/p} \Big\|_{C^1(\T^d)} \|\Theta_\mu^j\|_{L^p}.
\end{split}
\end{equation*}
Notice now that
\begin{equation*}
\bigg\|\psi(t) \chi_j(t) \sign \Big( R_{0,j}(t) \Big) \big|R_{0,j}(t)\big|^{1/p} \bigg\|_{L^p} \leq \Big\| \big|R_{0,j}(t)\big|^{1/p} \Big\|_{L^p} \leq \|R_0(t)\|_{L^1}^{1/p}
\end{equation*}
and, recalling the definition of the cutoff $\chi_j$ in Section \ref{ss:def:pert},
\begin{equation*}
\Big\| \psi(t) \chi_j(t) \sign ( R_{0,j}(t)) \big|R_{0,j}(t)\big|^{1/p} \Big\|_{C^1(\T^d)}
\leq C\big(\delta, \|R_0(t)\|_{C^1(\T^d)} \big).
\end{equation*}
Therefore, using the bounds on $\|\Theta_\mu^j\|_{L^p}$ provided in \eqref{eq:mikado:est:1}, we get
\begin{equation*}
\|\vartheta(t)\|_{L^p} \leq \frac{M}{2}\eta  \|R_0(t)\|_{L^1}^{1/p} + \frac{C(\eta, \delta, \|R_0(t)\|_{C^1})}{\lambda^{1/p}}.
\end{equation*}

\end{proof}

\begin{lemma}[Estimate on $\vartheta_c$]
\label{l:vartheta:c}
It holds
\begin{equation*}
|\vartheta_c(t)| \leq \frac{C(\eta, \|R_0(t)\|_{C^1(\T^d)})}{\lambda}.
\end{equation*}
\end{lemma}
\begin{proof}
We use Lemma \ref{l:mean:value}:
\begin{equation*}
\begin{split}
|\vartheta_c(t)| 	& \leq \eta \sum_{j=1}^d \frac{\sqrt{d} \|R_0(t)\|_{C^1(\T^d)} \|\Theta_\mu^j\|_{L^1}}{\lambda} \\
					& \leq \frac{C(\eta, \|R_0(t)\|_{C^1(\T^d)})}{\lambda}.
\end{split}
\end{equation*}
\end{proof}
%

\begin{lemma}[$L^{p'}$ norm of $w$]
\label{l:lp:w}
For every time $t \in [0,1]$,
\begin{equation*}
\|w(t)\|_{L^{p'}(\T^d)} \leq \frac{M}{2\eta} \|R_0(t)\|^{1/p'}_{L^1(\T^d)} +  \frac{C(\eta, \delta, \|R_0(t)\|_{C^1(\T^d)})}{\lambda^{1/p'}}.
\end{equation*}
\end{lemma}
\begin{proof}
The proof is completely analogous to the proof of Lemma \ref{l:lp:vartheta}, with $\eta^{-1}$ instead of $\eta$ and $\|W_\mu^j\|_{L^{p'}}$ instead of $\|\Theta_\mu^j\|_{L^p}$, and thus it is omitted.  
\end{proof}

\begin{lemma}[$W^{1,\tilde p}$ norm of $w$]
\label{l:w1p:w}
For every time $t \in [0,1]$, 
\begin{equation*}
\|w(t)\|_{W^{1, \tilde p}(\T^d)} \leq C \Big(\eta,  \|R_0\|_{C^1} \Big) \lambda \mu^{-\gamma}.
\end{equation*}
\end{lemma}
\begin{proof}
We have
\begin{equation*}
Dw(t,x) = \eta^{-1} \psi(t) \sum_{j=1}^d   W_\mu^j (\lambda x) \otimes D \Big( \chi_j |R_{0,j}|^{1/p'} \Big)
+ \lambda \chi_j |R_{0,j}|^{1/p'} D W_\mu^j(\lambda x),
\end{equation*}
from which we get the pointwise estimate
\begin{equation*}
|Dw(t,x)| \leq C\big(\eta, \delta, \|R_0\|_{C^1} \big) \sum_{j=1}^d \bigg(|W_\mu^j(\lambda x)| + \lambda |D W_\mu^j(\lambda x)| \bigg).
\end{equation*}
We can take now the $L^{\tilde p}$ norm of $Dw(t)$ and use \eqref{eq:mikado:est:2} to get
\begin{equation*}
\begin{split}
\|Dw(t)\|_{L^{\tilde p}} 
	& \leq C \big(\eta, \delta, \|R_0\|_{C^1} \big) \sum_{j=1}^d \bigg( \|W_\mu^j\|_{L^{\tilde p}} + \lambda \|DW_\mu^j\|_{L^{\tilde p}} \bigg) \\
	& \leq C \big(\eta, \delta, \|R_0\|_{C^1} \big) \lambda \mu^{-\gamma}.
\end{split}
\end{equation*}
A similar (and even easier) computation holds for $\|w(t)\|_{L^{\tilde p}}$, thus concluding the proof of the lemma. 
\end{proof}

\begin{lemma}[$L^{p'}$ norm of $w_c$]
\label{l:lp:wc}
For every time $t \in [0,1]$, 
\begin{equation*}
\|w_c(t)\|_{L^{p'}(\T^d)} \leq \frac{C (\eta, \delta,  \|R_0\|_{C^2} )}{\lambda}.
\end{equation*}
\end{lemma}
\begin{proof}
The corrector $w_c$ is defined in \eqref{eq:wc} using the antidivergence operator of Lemma \ref{l:improved:antidiv}. We can thus use the bounds given by that lemma, with $k=0$, to get
\begin{equation*}
\begin{split}
\|w_c(t)\|_{L^{p'}} 
		& \leq \eta^{-1} \sum_{j=1}^d \frac{C_{0,p'}}{\lambda} \Big\| \nabla \big( \psi(t) \chi_j(t) |R_{0,j}(t)|^{1/p'} \big) \Big\|_{C^1(\T^d)} \big\|W_\mu^j\big\|_{L^{p'}} \\
		& \leq \frac{C(\eta, \delta, \|R_0\|_{C^2})}{\lambda} \sum_{j=1}^d \|W_\mu^j\|_{L^{p'}} \\
\text{(by \eqref{eq:mikado:est:1})}
		& \leq \frac{C(\eta, \delta, \|R_0\|_{C^2})}{\lambda}.
\end{split}
\end{equation*}

\end{proof}

\begin{lemma}[$W^{1,\tilde p}$ norm of $w_c$]
For every time $t \in [0,1]$, 
\begin{equation*}
\|w_c(t)\|_{W^{1, \tilde p}(\T^d)} \leq C \Big(\eta, \|R_0\|_{C^3} \Big) \mu^{-\gamma}.
\end{equation*}
\end{lemma}
\begin{proof}
We estimate only $\|Dw_c(t)\|_{L^{\tilde p}}$, the estimate for $\|w_c(t)\|_{L^{\tilde p}}$ is analogous and even easier. We use once again the bounds provided by Lemma \ref{l:improved:antidiv} with $k=1$:
\begin{equation*}
\begin{split}
\|Dw_c(t)\|_{L^{\tilde p}} 
		& \leq \eta^{-1} C_{1,\tilde p} \sum_{j=1}^d  \Big\| \nabla \big( \psi(t) \chi_j(t) |R_{0,j}(t)|^{1/p} \big) \Big\|_{C^2(\T^d)} \big\|W_\mu^j\big\|_{W^{1,\tilde p}} \\
\text{(by \eqref{eq:mikado:est:2})}
		& \leq C(\eta, \delta, \|R_0\|_{C^3}) \mu^{-\gamma}.
\end{split}
\end{equation*}
\end{proof}

\section{The new defect field}
\label{s:reynolds}

In this section we continue the proof of Proposition \ref{p:main}, defining the new defect field $R_1$ and proving some estimates on it.

\subsection{Definition of the new defect field}

We want to define $R_1$ so that
\begin{equation*}
-\div R_1 = \partial_t \rho_1 + \div(\rho_1 u_1).
\end{equation*}
Let us compute
\begin{equation}
\label{eq:new:rest:computation}
\begin{split}
\partial_t \rho_1 + \div(\rho_1 u_1) 
	& = \div (\vartheta w - R_0) \\
	& \quad + \partial_t (\vartheta + \vartheta_c) + \div( \vartheta u_0 + \rho_0 w ) \\
	& \quad + \div(\rho_0 w_c + \vartheta_c u_0 + \vartheta w_c + \vartheta_c w + \vartheta_c w_c)  \\
	& = \div \Big[ (\vartheta w - R_0) \\
	& \quad \quad \qquad + \big(\nabla \Delta^{-1} \partial_t (\vartheta + \vartheta_c) + \vartheta u_0 + \rho_0 w  \big)  \\
	& \quad \quad \qquad + \big( \rho_0 w_c  + \vartheta_c u_0 + \vartheta w_c + \vartheta_c w + \vartheta_c w_c \big) \Big]  \\
	& = \div \Big[ (\vartheta w - R_0) + R^{\rm linear} + R^{\rm corr} \Big]  \\
\end{split}
\end{equation}
where we put 
\begin{equation}
\label{eq:rlin:rcorr}
\begin{split}
R^{\rm linear} 	& :=  \nabla \Delta^{-1} \partial_t ( \vartheta + \vartheta_c ) + \vartheta u_0 + \rho_0 w  \\
R^{\rm corr} 	& := \rho_0 w_c  + \vartheta_c u_0 + \vartheta w_c + \vartheta_c w + \vartheta_c w_c .
\end{split}
\end{equation}
Note that we can apply the antidivergence operator $\nabla \Delta^{-1}$ to $\partial_t (\vartheta + \vartheta_c)$, since it has zero mean value. Let us now consider the term $\vartheta w - R_0$. Recall from Proposition \ref{p:mikado}, that, for $j \neq k$, $\supp \Theta_\mu^j \cap \supp W_\mu^k = \emptyset$. Coherent with \eqref{eq:periodic}, we use the notation
\begin{equation*}
(\Theta_\mu^j W_\mu^j)_\lambda(x) = \Theta_\mu^j(\lambda x) W_\mu^j(\lambda x).
\end{equation*}
We have
\begin{equation*}
\begin{split}
\vartheta(t) w(t) - R_0(t)
	& = \sum_{j=1}^d \psi^2(t) \chi_j^2(t) R_{0,j}(t) (\Theta_\mu^j W_\mu^j)_\lambda - R_0(t)\\
	& = \sum_{j=1}^d \psi^2(t) \chi_j^2(t) R_{0,j}(t) \big[(\Theta_\mu^j W_\mu^j)_\lambda - e_j \big] \\
	& \qquad + \psi^2(t) \sum_{j=1}^d \big[ \chi_j^2(t) - 1 \big] R_{0,j}(t) e_j \\
	& \qquad + \big[ \psi^2(t) - 1 \big] R_0(t)\\
	& = \sum_{j=1}^d \psi^2(t) \chi_j^2(t) R_{0,j}(t) \big[ (\Theta_\mu^j W_\mu^j)_\lambda - e_j \big] \\
	& \qquad + R^\chi(t) + R^\psi(t),
\end{split}
\end{equation*}
where we put
\begin{equation}
\label{eq:rchi:rpsi}
\begin{split}
R^\chi(t) 	&:= \psi^2(t) \sum_{j=1}^d \big[ \chi_j^2(t) - 1 \big] R_{0,j}(t) e_j,  \\
R^\psi(t) 	&:= \big[ \psi^2(t) - 1 \big] R_0(t).
\end{split}
\end{equation}
Thus, using again Proposition \ref{p:mikado}, and in particular the fact that $\div (\Theta_\mu^j W_\mu^j) = 0$, we get
\begin{equation}
\label{eq:rest:quadr}
\begin{split}
\div (\vartheta(t) w(t) - R_0(t)) 	& = \sum_{j=1}^d \nabla \Big(\psi^2(t) \chi_j^2(t) R_{0,j}(t) \Big) \cdot \Big[ (\Theta_\mu^j W_\mu^j)_\lambda - e_j \Big] \\
									& \qquad + \div (R^\chi + R^\psi).
\end{split}
\end{equation}
Each term in the summation over $j$ has the form $f \cdot g_\lambda$ and it has zero mean value, being a divergence. Moreover, again by Proposition \ref{p:mikado},
\begin{equation*}
\fint_{\T^d} (\Theta_\mu^j W_\mu^j)_\lambda dx = \fint_{\T^d} \Theta_\mu^j W_\mu^j dx = e_j. 
\end{equation*}
Therefore we can apply Lemma \ref{l:improved:antidiv} and define
\begin{equation}
\label{eq:R:quadr}
R^{\rm quadr}(t) := \sum_{j=1}^d \mathcal R \bigg(\nabla  \Big( \psi^2(t) \chi_j^2(t) R_{0,j}(t) \Big) \cdot \Big[ (\Theta_\mu^j W_\mu^j)_\lambda - e_j \Big] \bigg).
\end{equation} 
By Remark \ref{rmk:time:as:param}, $R^{\rm quadr}$ is smooth in $(t,x)$. Summarizing, from \eqref{eq:new:rest:computation} and \eqref{eq:rest:quadr} we get
\begin{equation*}
\partial_t \rho_1 + \div(\rho_1 u_1) = \div \Big[R^{\rm quadr} + R^\chi + R^\psi + R^{\rm linear} + R^{\rm corr} \Big] .  
\end{equation*}
We thus define
\begin{equation}
\label{eq:reynolds}
- R_1 : = R^{\rm quadr}  + R^\chi + R^\psi + R^{\rm linear} + R^{\rm corr}.
\end{equation}

\noindent Aim of the next section will be to get an estimate in $L^1$ for $R_1(t)$, by estimating separately each term in \eqref{eq:reynolds}.

\subsection{Estimates on the defect field}

We now prove some estimates on the different terms which define $R_1$. 
%
%


\begin{lemma}[Estimate on $R^{\rm quadr}$]
\label{l:rquadr}
For every $t \in [0,1]$,
\begin{equation*}
\|R^{\rm quadr}(t)\|_{L^1(\T^d)} \leq \frac{C(\delta, \|R_0\|_{C^2})}{\lambda}.
\end{equation*}
\end{lemma}
\begin{proof}
$R^{\rm quadr}$ is defined in \eqref{eq:R:quadr} using Lemma \ref{l:improved:antidiv}. Observe first that
\begin{equation*}
\big\|\nabla  \big( \psi^2(t) \chi_j^2(t) R_{0,j}(t)\big)\big\|_{C^1(\T^d)} \leq C(\delta, \|R_0\|_{C^2} ).
\end{equation*}
Applying the bounds provided by Lemma \ref{l:improved:antidiv}, with $k=0$, and \eqref{eq:mikado:est:1} we get
\begin{equation*}
\begin{split}
\|R^{\rm quadr}& (t)\|_{L^1(\T^d)} \\
									& \leq \sum_{j=1}^d \frac{C_{0,1}}{\lambda} \big\|\nabla  \big( \psi^2(t) \chi_j^2(t) R_{0,j}(t)\big)\big\|_{C^1} \big\|\Theta_\mu^j W_\mu^j - e_j\big\|_{L^1(\T^d)} \\
									& \leq \sum_{j=1}^d \frac{C(\delta, \|R_0\|_{C^2})}{\lambda}. \qedhere
\end{split}
\end{equation*}
\end{proof}

\begin{lemma}[Estimate on $R^{\chi}$]
\label{l:rchi}
For every $t \in [0,1]$
\begin{equation*}
\|R^\chi(t)\|_{L^1(\T^d)} \leq \frac{\delta}{2}
\end{equation*}
\end{lemma}
\begin{proof}
Notice that $\chi_j(t,x) = 1$ if $|R_{0,j}(t,x)| \geq \delta/(2d)$. Therefore $R^\chi(t,x) \neq 0$ only when $|R_{0,j}(t,x)| \leq \delta/(2d)$. We thus have the pointwise estimate
\begin{equation*}
|R^\chi(t,x)| \leq \sum_{j=1}^d |\chi_j(t,x)^2 - 1| |R_{0,j}(t,x)|
\leq \frac{\delta}{2}.
\end{equation*}
from which the conclusion easily follows. 
\end{proof}

\begin{lemma}[Estimate on $R^{\psi}$]
\label{l:rpsi}
It holds
\begin{equation*}
\|R^\psi(t)\|_{L^1(\T^d)} 
\leq 
\begin{cases}
0,														& t \in I_\sigma, \\
\|R_0(t)\|_{L^1(\T^d)},						 	& t \in [0,1] \setminus I_\sigma.
\end{cases}
\end{equation*}
\end{lemma}

\begin{proof}
The proof follows immediately from the definition of $R^\psi$ in \eqref{eq:rchi:rpsi} and the definition of the cutoff $\psi$. 
\end{proof}

\begin{lemma}[Estimate on $R^{\rm linear}$]
\label{l:rlinear}
For every $t \in [0,1]$
\begin{equation*}
\|R^{\rm linear}(t)\|_{L^1(\T^d)} \leq  C\big(\eta, \delta, \sigma, \|\rho_0\|_{C^0}, \|u_0\|_{C^0}, \|R_0\|_{C^0}\big) \mu^{-\gamma}.
\end{equation*}
\end{lemma}
\begin{proof}
At each time $t \in [0,1]$, 
\begin{equation*}
\begin{split}
\|R^{\rm linear}&(t)\|_{L^1(\T^d)} \\
	& \leq \| \nabla \Delta^{-1} \partial_t ( \vartheta(t) + \vartheta_c(t) )\|_{L^1(\T^d)} + \|\vartheta(t) u_0(t)\|_{L^1(\T^d)}  + \|\rho_0(t) w(t)\|_{L^1(\T^d)}    \\
	& \leq \|\partial_t \vartheta(t)\|_{L^1(\T^d)} + |\vartheta_c'(t)| + \|\vartheta(t) u_0(t)\|_{L^1(\T^d)}  + \|\rho_0(t) w(t)\|_{L^1(\T^d)},
\end{split}
\end{equation*}
where the first term was estimating using Lemma \ref{l:antidiv}. We now separately estimate each term in the last sum.

\bigskip
\noindent \textit{1. Estimate on $\|\partial_t \vartheta(t)\|_{L^1}$}. We have
\begin{equation*}
\partial_t \vartheta(t) = \eta \sum_{j=1}^d \partial_t \Big( \psi(t) \chi_j(t,x) \sign(R_{0,j}(t,x)) |R_{0,j}(t,x)|^{1/p} \Big) \Theta_\mu^j(\lambda x)
\end{equation*}
from which we get the pointwise estimate
\begin{equation*}
|\partial_t \vartheta(t)| \leq C(\eta, \delta, \sigma, \|R_0\|_{C^1}) \sum_{j=1}^d |\Theta_\mu^j(\lambda x)|.
\end{equation*}
Using \eqref{eq:mikado:est:1}, we deduce
\begin{equation*}
\|\partial_t \vartheta(t)\|_{L^1} \leq C(\eta, \delta, \sigma, \|R_0\|_{C^1}) \mu^{-\gamma}. 
\end{equation*}

\bigskip
\noindent \textit{2. Estimate on $|\vartheta_c'(t)|$}. We have
\begin{equation*}
|\vartheta_c'(t)| \leq \|\partial_t \vartheta(t)\|_{L^1}  \leq C(\eta, \delta, \sigma, \|R_0\|_{C^1}) \mu^{-\gamma}.
\end{equation*}

\bigskip
\noindent \textit{3. Estimate on $\|\vartheta(t) u_0(t)\|_{L^1}$}. We now use the classical H\"older inequality to estimate
\begin{equation*}
\begin{split}
\|\vartheta(t) u_0(t)\|_{L^1} 
		& \leq \|u_0\|_{C^0} \|\vartheta(t)\|_{L^1} \\
		& \leq \eta \|u_0\|_{C^0} \sum_{j=1}^d \big\| |R_{0,j}|^{1/p} \big\|_{C^0} \|\Theta_\mu^j\|_{L^1} \\
\text{(by \eqref{eq:mikado:est:2})}
		& \leq C\big(\eta, \|u_0\|_{C^0}, \|R_0\|_{C^0}\big) \mu^{-\gamma}.  
\end{split}
\end{equation*}

\bigskip
\noindent \textit{4. Estimate on $\|\rho_0(t) w(t)\|_{L^1}$}. Similarly, again using the classical H\"older inequality,
\begin{equation*}
\begin{split}
\|\rho_0(t) w(t)\|_{L^1} 
		& \leq \|\rho_0\|_{C^0} \|w(t)\|_{L^1} \\
		& \leq \eta^{-1} \|\rho_0\|_{C^0} \sum_{j=1}^d \big\| |R_{0,j}|^{1/p'} \big\|_{C^0} \|W_\mu^j\|_{L^1} \\
\text{(by \eqref{eq:mikado:est:2})}
		& \leq C\big(\eta, \|\rho_0\|_{C^0}, \|R_0\|_{C^0}\big) \mu^{-\gamma}.  \qedhere
\end{split}
\end{equation*}
\end{proof}

\begin{lemma}[Estimate on $R^{\rm corr}$]
\label{l:rcorr}
For every $t \in [0,1]$,
\begin{equation*}
\|R^{\rm corr}(t)\|_{L^1(\T^d)} \leq \frac{C(\eta, \delta, \|\rho_0\|_{C^0}, \|u_0\|_{C^0}, \|R_0\|_{C^2})}{\lambda}.
\end{equation*}
\end{lemma}
\begin{proof}
We estimate separately each term in the definition \eqref{eq:rlin:rcorr} of $R^{\rm corr}$.

\bigskip
\noindent \textit{1. Estimate on $\rho_0 w_c$.} By the classical H\"older inequality, 
\begin{equation*}
\begin{split}
\|\rho_0(t) w_c(t)\|_{L^1} 	& \leq \|\rho_0\|_{C^0} \|w_c(t)\|_{L^1} \\
								& \leq \|\rho_0\|_{C^0} \|w_c(t)\|_{L^{p'}} \\
\text{(by Lemma \ref{l:lp:wc})}
								& \leq \frac{C(\eta, \delta, \|\rho_0\|_{C^0}, \|R_0\|_{C^2}) }{\lambda}. 
\end{split}
\end{equation*}

\bigskip
\noindent \textit{2. Estimate on $\vartheta_c u_0$.} We use Lemma \ref{l:vartheta:c}:
\begin{equation*}
\begin{split}
\|\vartheta_c(t) u_0(t)\|_{L^1} 	\leq |\vartheta_c(t)| \|u_0\|_{C^0}  \leq \frac{C(\eta, \|u_0\|_{C^0}, \|R_0\|_{C^2})}{\lambda}.
\end{split}
\end{equation*}

\bigskip
\noindent \textit{3. Estimate on $\vartheta w_c$.} We use Lemma \ref{l:lp:vartheta} and Lemma \ref{l:lp:wc}:
\begin{equation*}
\begin{split}
\|\vartheta(t) w_c(t)\|_{L^1}  \leq \|\vartheta(t)\|_{L^p} \|w_c(t)\|_{L^{p'}}  \leq  \frac{C(\eta, \delta, \|R_0\|_{C^2})}{\lambda}.
\end{split}
\end{equation*}

\bigskip
\noindent \textit{4. Estimate on $\vartheta_c w$.} We use Lemma \ref{l:vartheta:c} and Lemma \ref{l:lp:w}:
\begin{equation*}
\begin{split}
\|\vartheta_c(t) w(t)\|_{L^1}
& \leq |\vartheta_c(t)| \|w(t)\|_{L^1}  \\
& \leq |\vartheta_c(t)| \|w(t)\|_{L^{p'}}  \\
& \leq \frac{C(\eta, \delta, \|R_0\|_{C^2})}{\lambda}.
\end{split}
\end{equation*}

\bigskip
\noindent \textit{5. Estimate on $\vartheta_c w_c$.} We use Lemma \ref{l:vartheta:c} and Lemma \ref{l:lp:wc}:
\begin{equation*}
\begin{split}
\|\vartheta_c(t) w_c(t)\|_{L^1}
& \leq |\vartheta_c(t)| \|w_c(t)\|_{L^1} \\
& \leq |\vartheta_c(t)| \|w_c(t)\|_{L^{p'}} \\
& \leq \frac{C(\eta, \delta, \|R_0\|_{C^2})}{\lambda^2}.  \qedhere
\end{split}
\end{equation*}
\end{proof}

\begin{remark}
In estimating $R^{\rm corr}$ the only term where we really need the fast oscillation $\lambda$ is the estimate on $\vartheta w_c$. All the other terms could be alternatively estimated using the concentration parameter $\mu$, since, by \eqref{eq:mikado:est:2}, $|\vartheta_c(t)|, \|w_c(t)\|_{L^1}, \|w(t)\|_{L^1} \leq \text{const.} \ \mu^{-\gamma}$. In this way we would obtained the less refined estimate
\begin{equation*}
\|R^{\rm corr}(t)\|_{L^1} \leq \frac{C}{\lambda} + \frac{C}{\mu^\gamma},
\end{equation*}
which is however enough to prove Proposition \ref{p:main}.  
\end{remark}

\section{Proof of Proposition \ref{p:main}}
\label{s:proof:prop}

In this section we conclude the proof of Proposition \ref{p:main}, proving estimates \eqref{eq:dist:rho:stat}-\eqref{eq:reyn:stat}. We will choose $$\mu = \lambda^c$$ for a suitable $c>1$ and $\lambda$ sufficiently large. 

\bigskip
\noindent \textit{1. Estimate \eqref{eq:dist:rho:stat}.} We have
\begin{equation*}
\begin{split}
\|\rho_1(t) - \rho_0(t)\|_{L^p}
& \leq \|\vartheta_0(t)\|_{L^p} + |\vartheta_c(t)| \\
\text{(Lemmas \ref{l:lp:vartheta} and \ref{l:vartheta:c})} 
& \leq \frac{M}{2} \eta \|R_0(t)\|^{1/p}_{L^1} +  \frac{C(\eta, \delta, \|R_0(t)\|_{C^1})}{\lambda^{1/p}} \\
& \qquad + \frac{C(\eta, \|R_0(t)\|_{C^1})}{\lambda} \\
& \leq M \eta \|R_0(t)\|^{1/p}_{L^1},
\end{split}
\end{equation*}
if the constant $\lambda$ is chosen large enough. Notice also that, if $t \in [0,1] \setminus I_{\sigma/2}$, then $\vartheta(t) \equiv 0$ and $\vartheta_c(t) = 0$, thanks to the cutoff $\psi$ in \eqref{eq:perturbation}. Therefore \eqref{eq:dist:rho:stat} is proven. 

\bigskip
\noindent \textit{2. Estimate \eqref{eq:dist:u:1:stat}.} The estimate uses Lemmas \ref{l:lp:w} and \ref{l:lp:wc} and it is completely similar to what we just did for \eqref{eq:dist:rho:stat}.

\bigskip
\noindent \textit{4. Estimate \eqref{eq:dist:u:2:stat}.}  By Lemma \ref{l:w1p:w}, 
\begin{equation*}
\|w(t)\|_{W^{1, \tilde p}} \leq C \Big(\eta, \|R_0\|_{C^1} \Big) \lambda \mu^{-\gamma} \leq \delta,
\end{equation*}
if $\mu$ is chosen of the form $\mu = \lambda^c$ with $c > 1/\gamma$ and $\lambda$ is chosen large enough. 

\bigskip
\noindent \textit{4. Estimate \eqref{eq:reyn:stat}.} 
Recall the definition of $R_1$ in \eqref{eq:reynolds}. 
Using Lemmas \ref{l:rquadr},  \ref{l:rchi}, \ref{l:rpsi}, \ref{l:rlinear}, \ref{l:rcorr}, for $t \in I_\sigma$, we get
\begin{equation*}
\begin{split}
\|R_1(t)\|_{L^1} 	& \leq \frac{\delta}{2} + C(\eta, \delta, \sigma, \|\rho_0\|_{C^0}, \|u_0\|_{C^0}, \|R_0\|_{C^2}) \bigg( \frac{1}{\lambda} + \frac{1}{\mu^\gamma} \bigg)  \\
					& \leq \frac{\delta}{2} + C(\eta, \delta, \sigma, \|\rho_0\|_{C^0}, \|u_0\|_{C^0}, \|R_0\|_{C^2}) \bigg( \frac{1}{\lambda} + \frac{1}{\lambda^{c \gamma}} \bigg)  \\
					& \leq \delta
\end{split}
\end{equation*}
provided $\lambda$ is chosen large enough. Similarly, for $t \in I_{\sigma/2}  \setminus I_\sigma$, we have
\begin{equation*}
\begin{split}
\|& R_1(t)\|_{L^1} \\
& \leq \|R_0(t)\|_{L^1} + \frac{\delta}{2} + C(\eta, \delta, \sigma, \|\rho_0\|_{C^0}, \|u_0\|_{C^0}, \|R_0\|_{C^2}) \bigg( \frac{1}{\lambda} + \frac{1}{\mu^{\gamma}} \bigg)  \\
& \leq \|R_0(t)\|_{L^1} + \frac{\delta}{2} + C(\eta, \delta, \sigma, \|\rho_0\|_{C^0}, \|u_0\|_{C^0}, \|R_0\|_{C^2}) \bigg( \frac{1}{\lambda} + \frac{1}{\lambda^{c \gamma}} \bigg)  \\
& \leq \|R_0(t)\|_{L^1} + \delta
\end{split}
\end{equation*}
if $\lambda$ is chosen large enough. Finally, for $t \in [0,1] \setminus I_{\sigma/2}$, the cutoff function $\psi(t) \equiv 0$, and thus $\vartheta(t) = \vartheta_c(t) = w(t) = w_c(t) = 0$. Therefore $R_1(t) = R_0(t)$.

\section{Sketch of the proofs of Theorems \ref{thm:strong}, \ref{thm:diffusion}, \ref{thm:diffusion:higher}}
\label{s:sketch}

Theorems \ref{thm:strong}, \ref{thm:diffusion}, \ref{thm:diffusion:higher} can be proven in a very similar way to Theorem \ref{thm:main} and thus we present just a sketch of their proofs. 

The proof of Theorem \ref{thm:main} follows from Proposition \ref{p:main}: exactly in the same way, for each one of Theorems \ref{thm:strong}, \ref{thm:diffusion}, \ref{thm:diffusion:higher} there is a corresponding \emph{main} proposition, from which the proof the theorem follows.

\subsection{Sketch of the proof of Theorem \ref{thm:diffusion}}

The proof of Theorem \ref{thm:diffusion} follows from the next proposition, exactly in the same way as Theorem \ref{thm:main} follows from Proposition \ref{p:main}. Let us consider the equation 
\begin{equation}\label{eq:cont:diffu}
\left\{
\begin{split}
\partial_t \rho + \div(\rho u) 	-\Delta\rho& = - \div R, \\
\div u 								& = 0.
\end{split}
\right.	
\end{equation}
\begin{proposition}
\label{p:main:diffusion}
Proposition \ref{p:main} holds with \eqref{eq:cont:diffu} instead of \eqref{eq:cont:reyn}. 
\end{proposition}
\begin{proof}[Sketch of the proof of Proposition \ref{p:main:diffusion}]
Exactly as in the proof of Proposition \ref{p:main}, we define the Mikado densities and fields as in Proposition \ref{p:mikado} and we choose the exponents $a,b$ as in Remark \ref{rmk:choice:ab}. We observe that, in addition to \eqref{eq:mikado:est:1}, \eqref{eq:mikado:est:2}, it also holds
\begin{equation}
\label{eq:mikado:est:2:d}
\|\nabla \Theta_\mu^j\|_{L^1} \leq M \mu^{-\gamma_4} \leq M \mu^{-\gamma}
\end{equation}
for 
\begin{equation*}
\gamma = \min\big\{\gamma_1, \gamma_2, \gamma_3, \gamma_4 \big\} > 0,
\end{equation*}
where $\gamma_1, \gamma_2, \gamma_3$ were defined in Remark \ref{rmk:choice:ab} and 
\begin{equation*}
\gamma_4 := \frac{d-1}{p'} - 1 > 0,
\end{equation*}
because of the second condition in \eqref{eq:p:ineq:diffu}. Then the perturbations $\vartheta, w, \vartheta_c, w_c$ can be defined as in Section \ref{ss:def:pert} and the estimates in Section \ref{ss:estimates-perturbation} continue to hold. In the definition of the new defect field in Section \ref{s:reynolds} we want to define $R_1$ so that
$$
-\div R_1=\partial_t\rho_1+\div(\rho_1u_1)-\Delta\rho_1,
$$
which leads to an additional term $\nabla\theta$ in the expression for $R^{\rm linear}$ in \eqref{eq:rlin:rcorr}. As a consequence the only estimate which changes is Lemma \ref{l:rlinear}. From \eqref{eq:mikado:est:2:d} and the expression for $\theta$ in \eqref{eq:perturbation} we easily obtain
$$
\|\nabla\theta\|_{L^1(\T^d)}\leq C(\eta,\delta,\|R_0\|_{C^1})\lambda\mu^{-\gamma}.
$$
Since we choose $\mu=\lambda^c$ with $c>1/\gamma$ in Section \ref{s:proof:prop}, the final estimates for $\|R_1(t)\|_{L^1}$ continue to hold. This concludes the proof of proposition (and thence also the proof of Theorem \ref{thm:diffusion}).  
\end{proof}

\subsection{Sketch of the proof of Theorem \ref{thm:strong}}

Also for Theorem \ref{thm:strong} there is a \emph{main} proposition, analog to Proposition \ref{p:main}.
\begin{proposition}
\label{p:main:strong}
There exists a constant $M>0$ and an exponent $s \in (1, \infty)$ such that  
the following holds. Let $\eta, \delta, \sigma>0$ and let $(\rho_0, u_0, R_0)$ be a smooth solution of the continuity-defect equation \eqref{eq:cont:reyn}. Then there exists another smooth solution $(\rho_1, u_1, R_1)$ of \eqref{eq:cont:reyn} such that
\begin{subequations}
\begin{align}
\|\rho_1(t) - \rho_0(t)\|_{L^s(\T^d)}		& \leq 
\begin{cases}
M \eta \|R_0(t)\|^{1/s}_{L^1(\T^d)}, 		& t \in I_{\sigma/2},  \\
0, 											& t \in [0,1] \setminus I_{\sigma/2},
\end{cases} \label{eq:strong:dist:rho:stat} \\
\|u_1(t) - u_0(t)\|_{L^{s'}(\T^d)}			& \leq 
\begin{cases}
M \eta^{-1} \|R_0(t)\|^{1/s'}_{L^1(\T^d)}, 		& t \in I_{\sigma/2}, \\
0, 													& t \in [0,1] \setminus I_{\sigma/2}, 
\end{cases} \label{eq:strong:dist:u:1:stat} \\
\left.
\begin{aligned}
\|\rho_1 (t) -  \rho_0(t)\|_{W^{m,p}(\T^d)} \\[.5em] \|u_1(t) - u_0(t)\|_{W^{\tilde m,\tilde p}(\T^d)}
\end{aligned} 
\right\}
& \leq \ \ \delta, \label{eq:strong:dist:u:2:stat} \\[.1em]
\|R_1(t)\|_{L^1(\T^d)} 						& \leq
\begin{cases}
\delta,										& t \in I_\sigma, \\
\|R_0(t)\|_{L^1(\T^d)} + \delta,			& t \in I_{\sigma/2} \setminus I_\sigma, \\
\|R_0(t)\|_{L^1(\T^d)},					& t \in [0,1] \setminus I_{\sigma/2}.
\end{cases} \label{eq:strong:reyn:stat}
\end{align}
\end{subequations}

\end{proposition}

\noindent Theorem \ref{thm:strong} can be deduced from Proposition \ref{p:main:strong} exactly in the same way as Theorem \ref{thm:main} was deduced from Proposition \ref{p:main}. The only difference here is the following. In general, it is not true that $\rho(t) \in L^p(\T^d)$, $u(t) \in L^{p'} (\T^d)$.  Therefore the fact that $\rho u \in C((0,T); L^1(\T^d))$ is proven by showing that $\rho \in C((0,T); L^s(\T^d))$ (thanks to  \eqref{eq:strong:dist:rho:stat}) and $u \in C((0,T); L^{s'} (\T^d))$ (thanks to \eqref{eq:strong:dist:u:1:stat}).
%

\begin{proof}[Sketch of the proof of Proposition \ref{p:main:strong}]
The proof is analog to the proof of Proposition \ref{p:main} presented in Sections \ref{s:perturbations}-\ref{s:proof:prop}. Here, however, we need modify the ``rate of concentration'' of the Mikado fields defined in \eqref{eq:non:periodic} to achieve better estimates on the derivatives. In other words, we have to modify the choice of $a,b$ in \eqref{eq:non:periodic}, as follows.
In order to get estimates \eqref{eq:strong:dist:rho:stat}-\eqref{eq:strong:dist:u:1:stat}, we want
\begin{equation}
\label{eq:strong:mikado:est:1}
\begin{split}
\|\Theta_\mu^j\|_{L^s(\T^d)}, \   \|W_\mu^j\|_{L^{s'}(\T^d)}, 
& \leq \text{const.},
\end{split}
\end{equation}
to get \eqref{eq:strong:dist:u:2:stat} we want
\begin{equation}
\label{eq:strong:mikado:est:2}
\begin{split}
\|\Theta_\mu^j\|_{W^{m, p}} \leq \text{const} \cdot \mu^{-\gamma}, \ 
\|W_\mu^j\|_{W^{\tilde m, \tilde p}} 		\leq \text{const} \cdot \mu^{-\gamma}; 
\end{split}
\end{equation}
we also require
\begin{equation}
\label{eq:strong:mikado:est:3}
\begin{split}
\|\Theta_\mu^j\|_{L^1(\T^d)}, 		\leq \text{const} \cdot \mu^{-\gamma}, \ \  
\|W_\mu^j\|_{L^{1}(\T^d)}			\leq \text{const} \cdot \mu^{-\gamma}, \ \ 
\end{split}
\end{equation}
for some positive constant $\gamma >0$. 
Compare \eqref{eq:strong:mikado:est:1} with the first and the second estimates in \eqref{eq:mikado:est:1}, compare \eqref{eq:strong:mikado:est:2} with the last estimate in \eqref{eq:mikado:est:2} and compare \eqref{eq:strong:mikado:est:3} with the first and the second estimates in \eqref{eq:mikado:est:2}.

We want to find 
\begin{equation}
\label{eq:ab:bound}
a,b \in (0, d-1),
\end{equation}
so that $a+b = d-1$ and \eqref{eq:strong:mikado:est:2} is achieved. If we can do that, then condition \eqref{eq:strong:mikado:est:3} is a consequence of \eqref{eq:mikado:est} and \eqref{eq:ab:bound}. Similarly, condition \eqref{eq:strong:mikado:est:1} is automatically satisfied, choosing 
\begin{equation*}
s = \frac{d-1}{a}, \quad s' = \frac{d-1}{b}.
\end{equation*}
and observing that \eqref{eq:ab:bound} implies $s,s' \in (1,\infty)$. 

Using \eqref{eq:mikado:est}, we see that, to achieve \eqref{eq:strong:mikado:est:2}, we need
\begin{subequations}
\label{eq:system:2}
\begin{align}[left = \empheqlbrace\,]
a + m - \frac{d-1}{p} & < 0, 				\label{eq:system:2:a}	\\
b + \tilde m - \frac{d-1}{\tilde p} & < 0.	\label{eq:system:2:b}
\end{align}
\end{subequations}
Notice that, since $a+b = d-1$, \eqref{eq:system:2:b} is equivalent to 
\begin{equation*}
a > (d-1) \bigg( 1 - \frac{1}{\tilde p} \bigg) + \tilde m.
\end{equation*}
It is then possible to find $a,b$ satisfying \eqref{eq:ab:bound} and \eqref{eq:system:2}, with $a+d = d-1$, if and only if 
\begin{equation}
\label{eq:max:min}
\max \Bigg\{ 0, \ (d-1) \bigg( 1 - \frac{1}{\tilde p} \bigg) + \tilde m \Bigg\} < \min \Bigg\{d-1, \ \frac{d-1}{p} - m \Bigg\}
\end{equation}
and this last condition is equivalent to \eqref{eq:strong:p:ineq}.

Proposition \ref{p:main:strong} can now be proven exactly as we proved Proposition \ref{p:main} in Sections \ref{s:perturbations}-\ref{s:proof:prop}, this time using \eqref{eq:strong:mikado:est:1}-\eqref{eq:strong:mikado:est:3} instead of \eqref{eq:mikado:est:1}-\eqref{eq:mikado:est:2}. 
\end{proof}

\subsection{Sketch of the proof of Theorem \ref{thm:diffusion:higher}}

Once again, also for Theorem \ref{thm:diffusion:higher} there is a \emph{main} proposition, from which the proof of the theorem follows. Let us consider the equation
\begin{equation}
\label{eq:continuity:diffusion:defect}
\left\{
\begin{aligned}
\partial_t \rho + \div(\rho u) - L \rho & = - \div R, \\
\div u & = 0.
\end{aligned}
\right.
\end{equation}
Recall that $L$ is a constant coefficient differential operator of order $k \in \N$, $k \geq 2$. 
\begin{proposition}
\label{p:main:diffusion:higher}
Proposition \ref{p:main:strong} holds with \eqref{eq:continuity:diffusion:defect} instead of \eqref{eq:cont:reyn}. 
\end{proposition}
\begin{proof}[Sketch of the proof of Proposition \ref{p:main:diffusion:higher}]
Similarly to the proof of Proposition \ref{p:main:strong}, we want to choose the exponents $a,b \in (0, d-1)$ in Proposition \ref{p:mikado} so that \eqref{eq:strong:mikado:est:1}-\eqref{eq:strong:mikado:est:3} are satisfied and, moreover,
\begin{equation}
\label{eq:higher:diff}
\|D^{k-1} \Theta_\mu^j\|_{L^1(\T^d)} \leq \textrm{const} \cdot \mu^{-\gamma},
\end{equation}
for some positive constant $\gamma>0$. As in Proposition \ref{p:main:strong}, to get \eqref{eq:strong:mikado:est:1}-\eqref{eq:strong:mikado:est:3} we need \eqref{eq:max:min}. Moreover, condition \eqref{eq:higher:diff} is satisfied, provided
\begin{equation*}
a + (k-1) - (d-1) < 0,
\end{equation*}
or, equivalently,
\begin{equation}
\label{eq:higher:diff:a}
a < d - k.
\end{equation}
Putting together \eqref{eq:max:min} and \eqref{eq:higher:diff:a}, we obtain the condition
\begin{equation*}
\max \Bigg\{ 0, \ (d-1) \bigg( 1 - \frac{1}{\tilde p} \bigg) + \tilde m \Bigg\} < \min \Bigg\{d-1, \ d-k, \ \frac{d-1}{p} - m \Bigg\}.
\end{equation*}
It is now not difficult to see that the last inequality is satisfied if and only if \eqref{eq:p:ineq:diffu:2} holds. 

Then the perturbations $\vartheta, w, \vartheta_c, w_c$ can be defined as in Section \ref{ss:def:pert} and the estimates on the perturbations can be proven as in Proposition \ref{p:main:strong}. In the definition of the new defect field we want to define $R_1$ so that
$$
-\div R_1=\partial_t\rho_1+\div(\rho_1u_1)-L\rho_1.
$$
We can write $L = \div \tilde L$, where $\tilde L$ is a constant coefficient differential operator of order $k-1$. This leads to an additional term $\tilde L \theta$ in the expression for $R^{\rm linear}$ (compare with \eqref{eq:rlin:rcorr}), which can be estimated using \eqref{eq:higher:diff}:
$$
\|\tilde L \theta\|_{L^1(\T^d)}\leq C(\eta,\delta,\|R_0\|_{C^1})\lambda^{k-1} \mu^{-\gamma}.
$$
Choosing $\mu=\lambda^c$ with $c>(k-1)/\gamma$, we get the estimates for $\|R_1(t)\|_{L^1}$. This concludes the proof of the proposition (and thence also the proof of Theorem \ref{thm:diffusion:higher}).  
\end{proof}

\bibliographystyle{acm}
\bibliography{transport}

\end{document}